\documentclass[11
pt,a4paper]{article}
\usepackage[utf8]{inputenc}
\usepackage[english]{babel}
\usepackage[T1]{fontenc}
\usepackage{amsmath}
\usepackage{amsfonts}
\usepackage{amssymb}
\usepackage{amsthm} 
\usepackage[all]{xy} 
\usepackage[a4paper]{geometry}
\usepackage{mathrsfs}
\usepackage{frcursive}
\usepackage{color}
\usepackage{tikz-cd}
\usepackage{mathtools}
\usepackage{graphicx}

 \theoremstyle{plain}
 \newtheorem{thm}{Theorem}[section]
 \newtheorem{defi}[thm]{Definition}
 \newtheorem{prop}[thm]{Proposition}
 \newtheorem{lem}[thm]{Lemma}
 \newtheorem{coro}[thm]{Corollary}

 \theoremstyle{remark}
 \newtheorem{rem}[thm]{Remark}
 
 \newtheorem{ex}[thm]{Example}

\author{Sylvain Gaulhiac \thanks{Institut de Mathématiques de Jussieu-Paris Rive Gauche, Sorbonne-Université, Paris, France}}
\title{Towards tempered anabelian behaviour of Berkovich annuli }

\date{}

 \font\bf= cmbx10 at 10pt

 \newcommand{\ds}{\displaystyle}

 \newcommand{\A}{\mathbb A}
 \newcommand{\C}{\mathbb{C}}
 
 \newcommand{\G}{\mathbb{G}}

 \newcommand{\N}{\mathbb{N}}
 \newcommand{\PP}{\mathbb{P}}
 \renewcommand{\P}{\PP}
 \newcommand{\Q}{\mathbb{Q}}
 \newcommand{\R}{\mathbb{R}}
 \newcommand{\Z}{\mathbb{Z}}

\newcommand{\OO}{\mathscr{O}}

 \newcommand{\Aut}{\mathrm{Aut}}

\newcommand{\Gal}{\mathrm{Gal}}
 
 \newcommand{\Hom}{\mathrm{Hom}}
 \newcommand{\im}{\mathrm{im}}

\newcommand{\res}{\mathrm{res}}

\newcommand{\tp}{\mathrm{temp}}

\newcommand{\an}{\mathrm{an}}

\newcommand{\Cov}{\mathrm{Cov}}

\newcommand{\pp}{\, (p')}

\newcommand{\et}{\mathrm{\acute{e}t}}
\newcommand{\Kum}{\mathsf{Kum}}

\newcommand{\HC}{\mathscr{H}}
\newcommand{\MM}{\mathscr{M}}

 \newcommand{\GG}{\mathcal{G}}

 \newcommand{\CC}{\mathcal{C}}

 \newcommand{\iso}{\xrightarrow{\sim}}

\begin{document}

\maketitle

\renewcommand{\abstractname}{Abstract}
\begin{abstract}
This work brings to light some partial \emph{anabelian behaviours} of analytic annuli in the context of Berkovich geometry. More specifically, if $k$ is a valued non-archimedean complete field of mixed characteristic which is algebraically closed, and $\CC_1$, $\CC_2$ are two $k$-analytic annuli with isomorphic tempered fundamental group, we show that the lengths of $\CC_1$ and $\CC_2$ cannot be too far from each other. When they are finite, we show that the absolute value of their difference is bounded above with a bound depending only on the residual characteristic $p$.

\bigskip
\textbf{Keywords :} Anabelian geometry, Berkovich spaces, tempered fundamental group, analytic curves, cochain morphism, resolution of non-singularities.
\end{abstract}\

\tableofcontents\

\section*{Introduction}

Anabelian geometry is concerned with the following question : 

\begin{center}
\textit{To what extent is a geometric object determined by its fundamental group?}
\end{center}
It is within the framework of algebraic geometry that Grothendieck gave the first conjectures of anabelian geometry in a famous letter to Faltings in $1983$, where the fundamental group is nothing other than the étale one. 
Some deep results for hyperbolic curves have been obtained by Tamagawa and Mochizuki, answering certain conjectures of Grothendieck. However, almost no results are known for higher dimensions. 

\bigskip
In the context of Berkovich analytic geometry, it is possible to define several "fundamental groups" classifying for instance \emph{topological}, \emph{finite étale} or \emph{étale} (in the sens of \cite{DJg}) coverings. However, the group which seems to best capture anabelian behaviours of analytic spaces over non-archimedian fields is the \emph{tempered fundamental group}, introduced by Yves André in \cite{And}. This group classifies \emph{tempered coverings}, defined as étale coverings which become topological after finite étale base change. Both finite étale and topological coverings are examples of tempered coverings. 

\bigskip
It is Yves André, in \cite{And1}, who obtains for the first time some results of anabelian nature related to the tempered fundamental group. A few years later, a huge step was made in this direction with some results of Shinichi Mochizuki (\cite{M3}) followed by Emmanuel Lepage (\cite{Lep$1$} and \cite{Lep$2$}). These results relate the fundamental tempered group of the analytification of an algebraic hyperbolic curve to the dual graph of its stable reduction. If $X$ is a hyperbolic curve defined over some non-archimedian complete field $k$, the homotopy type of its analytification $X^\an$ can be described it terms of the stable model $\mathscr{X}$ of $X$. More precisely, if $\mathscr{X}_s$ stands for the special fibre of $\mathscr{X}$, the \emph{dual graph of the stable reduction} of $X$, denoted $\G_X$, is the finite graph whose vertices are the irreducible components of $\mathscr{X}_s$, and whose edges corresponds to the nodes (singularities in ordinary double points) between irreducible components. If $\overline{X}$ denotes the normal compactification of $X$, a cusp of $X$ is an element of $\overline{X}\setminus X$. Let us denote by $\G_X^\mathtt{c}$ the graph obtained from $\G_X$, adding one open edge to each cusp of $X$, call the \emph{extended dual graph of the stable reduction} of $X$. There exists a canonical topological embedding $\G_X^\mathtt{c}\hookrightarrow X^\an$ which admits a topologically proper deformation retraction $X^\an\twoheadrightarrow \G_X^\mathtt{c}$, thus $X^\an$ and $\G_X^\mathtt{c}$ have the same homotopy type.

\bigskip
Using the language of \emph{semi-graphs of anabelioids} and \emph{temperoids} introduced in high generality in \cite{M2} and \cite{M3}, Mochizuki proves in \cite{M3} that the fundamental tempered group of the analytification of a hyperbolic curve determines the dual graph of its stable reduction :

\begin{thm}[\cite{M3}, Corollary $3.11$]\label{théorème de Mochizuki}Let $X_1$ and $X_2$ be two hyperbolic curves over $\C_p$. Any outer isomorphism of groups $\varphi : \pi_1^\tp(X_1^\an) \iso \pi_1^\tp(X_2^\an)$ determines, functorially in $\varphi$, a unique isomorphism of graphs : $\overline{\varphi} : \G_{X_1}^\mathtt{c}\iso \G_{X_2}^\mathtt{c}$.

\end{thm}

Mochizuki shows more precisely that it is possible to reconstruct the graph of the stable reduction $\G_X$ of a hyperbolic curve $X$ from a $(p')$-version $\pi_1^{\tp,\pp}(X^\an)$ of the tempered fundamental group.

\bigskip
A few years later, Emmanuel Lepage refined this result. He proved that the knowledge of the tempered fundamental group of the analytification of a hyperbolic curve $X$ enables to not only reconstruct the graph $\G_X$, but also, in some cases, its canonical metric. This metric is such that the length of an edge corresponding to a node is the width of the annulus corresponding to the generic fibre of the formal completion on this node. It is, however, necessary to restrict this to \emph{Mumford curves}, which are defined as proper algebraic curves $X$ over $\C_p$ such that the normalized irreducible components of the stable reduction are isomorphic to $\P^1$. This is equivalent to saying in Berkovich language that the analytification $X^\an$ is locally isomorphic to open subsets of $\P^{1,\an}$, or that $X^\an$ does not contains any point of genus $>0$.

\begin{thm}[\cite{Lep$2$}]\label{théorème de Lepage}
Let $X_1$ and $X_2$ be two hyperbolic Mumford curves over $\C_p$, and $\varphi : \pi_1^\tp(X_1^\an) \iso \pi_1^\tp(X_2^\an)$ an isomorphism of groups. Then the isomorphism of graphs $\overline{\varphi} : \G_{X_1}\iso \G_{X_2}$ is an isomorphism of metric graphs.
\end{thm}

These two results deal with analytic curves which are of \emph{algebraic nature}, that is, analytifications of algebraic curves. Yet the theory of Berkovich analytic spaces is rich enough to contain lots of curves which are of \emph{analytic nature} without coming from algebraic curves. The most important examples of such curves, which are still very simple to define, are \emph{disks} and \emph{annuli}. In the wake of Mochizuki's and Lepage's results, one wonders whether similar anabelian results exist for more general analytic curves without imposing any algebraic nature. For such analytic curves, the generalisation of Mochizuki's results was carried out in the article \cite{Gau}, whereas the investigation about some analogous of Lepage's result is partially answered in this present article.

\bigskip

\textbf{Reconstruction of the analytic skeleton} For a quasi-smooth analytic curve $X$, the good analogous of the extended dual graph of the stable reduction is the \emph{analytic skeleton} $S^\an(X)$, defined in \ref{definition analytic skeleton}. When the skeleton meets all the connected components of $X$, there exists a canonical topological embedding $S^\an(X)\hookrightarrow X$ which admits a topologically proper deformation retraction $X\twoheadrightarrow S^\an(X)$. Therefore $X$ and $S^\an(X)$ have the same homotopy type. The restriction $S^\an(X)^\natural$ obtained from the skeleton by removing non-relatively compact edges is called the \emph{truncated skeleton} of $X$ (see \ref{definition truncated skeleton}), and is the analogous of the dual graph of the stable reduction. Let $k$ be a complete algebraically closed non-archimedean field of residual exponent $p$. In \cite{Gau}, $3.29$, is defined a certain of class of $k$-analytic curves, called \emph{$k$-analytically hyperbolic}. Their interest lies in the fact that for a $k$-analytically hyperbolic curve $X$ it is possible to reconstruct its truncated skeleton $S^\an(X)^\natural$ from the tempered group $\pi_1^\tp(X)$, or even from a \emph{prime-to-$p$} version $\pi_1^{\tp, \pp}(X)$, obtained by taking the projective limit of all quotients of $\pi_1^\tp(X)$ admitting a normal torsion-free subgroup of finite index prime to $p$. The reconstruction of $S^\an(X)^\natural$ from this group is given by the following : 
\begin{itemize}
\item[•]the vertices correspond to the conjugacy classes of maximal compact subgroups of $\pi_1^{\tp, \pp}(X)$;
\item[•]the edges correspond to the conjugacy classes of non-trivial intersections of two maximal compact subgroups of $\pi_1^{\tp, \pp}(X)$.
\end{itemize}

The condition for a quasi-smooth $k$-analytic curve to be analytically hyperbolic is stated in terms of non-emptiness of the sets of nodes of the skeleton and some combinatorial hyperbolic condition at each of these nodes. However, the analytical hyperbolicity may not be enough to recover all the skeleton. In order to recover also the non relatively compact edges of $S^\an(X)$ is defined in \cite{Gau}, $3.55$, a sub-class of $k$-analytically hyperbolic curves called \emph{$k$-analytically anabelian}. A $k$-analytically anabelian curve is a $k$-analytically hyperbolic curve satisfying a technical condition called \emph{ascendance vicinale}, which enables us to reconstruct open edges of the skeleton :

\begin{thm}[\cite{Gau}, 3.56]\label{théorème principal}
Let $X_1$ and $X_2$ be two $k$-analytically anabelian curves. Any group isomorphism $\varphi : \pi_1^\tp(X_1)\iso \pi_1^\tp(X_2)$ induces (functorially in $\varphi$) an isomorphism of semi-graphs between the analytic skeletons : $S^\an(X_1)\iso S^\an(X_2)$.

\end{thm}

\textbf{Anabelianity of length ? } This present article concentrates more on the potential anabelianity of lengths of edges of the skeleton of a $k$-analytic curve, inspired from the result of Lepage cited above. There is a natural way to define the length of an analytic annulus (see \ref{length of annuli}), invariant by automorphisms, which makes the skeleton $S^\an(X)$ of a quasi-smooth $k$-analytic curve $X$ a \emph{metric} graph. The question naturally arising is the following : 

\begin{center}
\textit{Does the tempered fundamental group $\pi_1^\tp(X)$ of a $k$-analytically anabelian curve $X$ determine $S^\an(X)$ as a metric graph?}
\end{center}

Before tackling the general case, it seems \textit{a priori} more simple to study first the case of a $k$-analytic annulus, even if this latter is not a $k$-analytically anabelian curve. The $\pp$-tempered group $\pi_1^{\tp, \pp}(\CC)$ of an annulus is always isomorphic to the $p'$-profinite completion $\widehat{\Z}^{\pp}$ of $\Z$, but its total tempered group $\pi_1^\tp(\CC)$ depends on its length whenever $k$ has mixed characteristic. The new question arising is the following : 

\begin{center}
\textit{Does the tempered group $\pi_1^\tp(\CC)$ of a $k$-analytic annulus $\CC$ determine its length?}
\end{center}

In order to investigate this question, one is tempted to follow the scheme of proof that Lepage develops in \cite{Lep$2$}. An idea would be to start from an "ovoid" $\mu_p$-covering of the annulus totally split at the middle of the skeleton, which would be analytically anabelian. Then knowing how to compute the length of any cycle would be enough to know the length of the annulus (by a limit argument). Yet one quickly faces problems of analytic nature that do not appear with Mumford curves : problems of detection of $\mu_{p^h}$-torsors with trivial $\Z/p^h\Z$-cochain. Indeed, if $Y\to X$ is a $\mu_n$-torsor, associating to some edge $e$ of $S^\an(X)$ the growth rate of any analytic function defining locally this torsor over $e$ leads to a harmonic cochain on the graph $S^\an(X)$ with values in $\Z/n\Z$. This growth rate corresponds to the degree of the strictly dominant monomial (see remark \ref{remarque conditions d'inversibilité}) of the corresponding analytic function. Therefore, when $X$ is a quasi-smooth $k$-analytic curve, we show in lemma \ref{existence du morphisme de cochaîne} that there exists a cochain morphism $\theta : H^1(X,\mu_n)\to \mathrm{Harm}(S^\an(X), \Z/n\Z)$ for any $n\in \N^\times$. However, when $n=p^h$ with $h>1$, it seems difficult to detect the kernel of $\theta$ from $\pi_1^\tp(X)$, which makes the hoped scheme of proof illusory. Nevertheless, the detection of $\ker(\theta)$ when $n=p$ is possible in some cases : 

\bigskip
\textbf{Théorème 0}
\textit{Let $X$ be a $k$-analytic  curve satisfying one of the two following conditions :}
\begin{enumerate}
\item \textit{$X$ is an annulus }
\item \textit{$X$ is a $k$-analytically hyperbolic curve, of finite skeleton without bridge, without any point of genus $>0$, without boundary, with only annular cusps, and such that there is never strictly more than one cusp coming from each node.}
\end{enumerate}

\textit{Then the set of $\mu_p$-torsors of $X$ with trivial $\Z/p\Z$-cochain, i.e. the $H^1(X,\mu_p)\cap \ker(\theta)$, is completely determined by $\pi_1^\tp(X)$.}\\

This result uses \emph{resolution of non-singularities} (section \ref{section résolution des non-singularités}) coupled with a characterisation of non-triviality of cochains in terms of minimality of splitting radius in rigid points (proposition \ref{minimalité des rayons de déploiement}). This characterisation can be rephrased set-theoretically with the splitting sets of torsors (corollary \ref{corollaire, détection de la nullité de la cochaîne sur une couronne}), that can themself be characterised from the tempered group by means of solvability of some "threshold" points (lemma \ref{lemme de résolubilité des points seuils}). \\

As for the initial question about the potential anabelianity of lengths of annuli, we found a partial answer, using the solvability of skeletons of annuli (lemma \ref{lemme de résolubilité du squelette d'une couronne}) doubled with some considerations of splitting sets of $\mu_p$-torsors :

\bigskip
\textbf{Théorème 1 : }
\textit{Let $\mathcal{C}_1$ and $\mathcal{C}_2$ be two $k$-analytic annuli whose tempered fundamental groups $\pi_1^\tp(\mathcal{C}_1)$ and $\pi_1^\tp(\mathcal{C}_2)$ are isomorphic. Then $\mathcal{C}_1$ has finite length if and only if $\mathcal{C}_2$ has finite length. In this case : 
$$\vert\ell(\mathcal{C}_1)-\ell(\mathcal{C}_2)\vert < \frac{2p}{p-1}.$$ We also have $d\left( \ell(\mathcal{C}_1),p\N^\times    \right)>1$ if and only if $d\left( \ell(\mathcal{C}_2),p\N^\times    \right)>1$, and in this case : $$\vert\ell(\mathcal{C}_1)-\ell(\mathcal{C}_2)\vert < \frac{p}{p-1}.$$}

\section{Berkovich analytic curves}

In all this text $k$ will denote a complete algebraically closed non-archimedean field of mixed characteristic $(0,p)$, i.e. $\mathrm{char}(k)=0$ and $\mathrm{char}(\widetilde{k})=p$, where $\widetilde{k}$ is the residue field of $k$. Let's assume that the absolute value on $k$ is normalized such that $\vert p\vert=p^{-1}$. The field $\C_p:=\widehat{\overline{\Q_p}}$ with the usual $p$-adic absolute value is an example of such field.

\subsection{Points and skeleton of an analytic curve}

A $k$-analytic curve is defined as a separated $k$-analytic space of pure dimension $1$. We send the reader to the fundational text \cite{Ber1} for references about analytic space, \cite{Ber2} for the cohomology on analytic spaces, and \cite{Duc} for a precise and systematic study of analytic curves. 

\bigskip
Let's recall here some properties about analytic curves which will be important in this text.\

Any $k$-analytic curve, endowed with the Berkovich topology, has very nice topological properties : locally compact, locally arcwise connected, and locally contractible, which makes possible to apply to it the usual theory of universal topological covering. Moreover, $k$-analytic curves are \emph{real graphs}, with potentially infinite branching, as stated by the following proposition : 

\begin{prop}[\cite{Duc}, $3.5.1$]\label{allure graphes reels}
Let $X$ be a non-empty connected $k$-analytic curve. The following statements are equivalent : \

\begin{itemize}
\item[i)] the topological space $X$ is contractible,
\item[ii)] the topological space $X$ is simply connected, 
\item[iii)] for any pair $(x,y)$ of points of $X$, there exists a unique closed subspace of $X$ homeomorphic to a compact interval with extremities $x$ and $y$. 
\end{itemize}
Moreover, any point of a $k$-analytic curve admits a basis of neighbourhood which are \emph{real trees}, i.e. satisfying the equivalent properties above. 
\end{prop}

\begin{rem}
Any real tree can be endowed with a topology called \emph{topology of real tree}, which might be different from its initial topology. The Berkovich topology on an open subset of an analytic curve which is a tree is coarser than the topology of real tree on this tree.  
\end{rem}

\bigskip
\textbf{Points of a $k$-analytic curve :} Let $X$ be a $k$-analytic curve, and $x\in X$. If $\mathscr{H}(x)$ denotes the completed residual field of $x$, it is possible to associate to the complete extension $\mathscr{H}(x)/k$ two transcendental values : 
\begin{align*}
f_x &=\mathrm{degtr}_{\widetilde{k}}\widetilde{\mathscr{H}(x)}\\
e_x &=\mathrm{rang}_\Q \left(  \vert \mathscr{H}(x)^\times\vert/\vert k^\times \vert \otimes_\Z \Q \right),
\end{align*} 
which satisfy $f_x+e_x\leqslant 1$ (in accordance with the Abhyankar inequality). The points of $X$ can be classified in 4 types according to the these transcendental values : 
\begin{defi}
A point $x\in X$ is : 
\begin{enumerate}
\item of \emph{type-$1$} if $\mathscr{H}(x)=k$ (in this case $f_x=e_x=0$)
\item of \emph{type $2$} if $f_x=1$
\item of \emph{type $3$} if $e_x=1$
\item of \emph{type $4$} if $f_x=e_x=0$ but $x$ is not of type $1$.
\end{enumerate}

For $i\in\lbrace 1,2,3,4 \rbrace$, let $X_{[i]}$ be the subset of $X$ consisting of type-$i$ points. 
\end{defi}

This definition of type-$1$ points holds here since we assumed that $k$ is algebraically closed. In general, a point $x\in X$ is of type $1$ if $\mathscr{H}(x)\subseteq\widehat{\overline{k}}$, where $\widehat{\overline{k}}$ denotes the completion of an algebraic closure of $k$. Since $k$ is algebraically closed, type-$1$ points are exactly the \emph{rigid} points, i.e. the points $x\in X$ such that the extension $\mathscr{H}(x)/k$ is finite. Since $k$ is by assumption non-trivially valued, $X_{[2]}$ is dense in $X$.\\

\textit{Preservation of type of points by finite morphisms.} If $f : X'\rightarrow X$ is a finite morphism of $k$-analytic curves, for any $i\in \lbrace 1,2,3,4\rbrace $, a point $x'\in X'$ is of type $i$ if and only if $f(x)$ is of type $i$.

\bigskip
One of the specificity of Berkovich geometry compared to ridid geometry is the existence of a \emph{boundary} which is embodied in the space. It is possible to define two boundaries of a $k$-analytic space : the \emph{analytic boundary} $\Gamma(X)$ and the \emph{Shilov boundary} $\partial^\an X$. However, specifically in the dimension $1$ case, i.e. for analytic curves, these two notions coincide, which allows to speak without any ambiguity about the \emph{boundary of $X$} $\partial^\an X\subseteq X$, potentially empty. 

\bigskip
\textbf{Description of the $k$-analytic affine line} $\A_k^{1, \an}$ : The analytification $\A_k^{1, \an}$ of the (algebraic) affine line $\A_k^1$ is the smooth, without boundary and connected $k$-analytic curve whose points are the multiplicative seminorms on the polynomial ring $k[T]$ extending the absolute value of $k$.
We are going to give an explicit description of $\A_k^{1, \an}$. For $r\geqslant 0$ and $a\in k$, let $B(a,r)=\lbrace x\in k, \vert x-a\vert \leqslant r \rbrace$ be the closed ball (which is also open since $k$ is non-archimedean!) of $k$, centred in $a$ and of radius $r$. 

\begin{itemize}
\bigskip
\item[•] Any element $a\in k$ determines a multiplicative seminorm on $k[T]$, the \emph{evaluation at $a$}, given by $P\in k[T]\mapsto \vert P(a)\vert$. It defines an element of $\A_k^{1,\an}$ denoted $\eta_a$, or $\eta_{a, 0}$. Then $\mathscr{H}(\eta_a)=k[T]/(T-a)\simeq k$, such that $\eta_a$ is a rigid point. 

\bigskip
\item[•] Let $a\in k$ et $r>0$. Consider the map : 
$$P\in k[T]\mapsto \sup_{b\in B(a,r)}\vert P(b)\vert=\sup_{b\in B(a,r)}\vert P\vert_{\eta_b}.$$
It actually defines an element of $\A_k^{1,\an}$, denoted $\eta_{a,r}$, and given by : $$\vert P (\eta_{a, r}) \vert=\mathrm{max}_{0\leqslant i \leqslant n}(\vert \alpha_i \vert r^i) \;\;\text{as soon as} \; \;P=\sum_{i=0}^n \alpha_i (T-a)^i.$$
One can verify that $\eta_{a, r}$ only depends on $B(a,r)$ (i.e. $\eta_{a,r}=\eta_{b,r}$ as soon as $b\in B(a,r)$). There are two cases :\

When $r\in \vert k^\times\vert$, $\widetilde{\mathscr{H}(\eta_{a,r})}=\widetilde{k}(T)$, and $\vert\mathscr{H}(\eta_{a,r})^\times\vert=\vert k^\times \vert$, such that $\eta_{a,r}$ is a type-$2$ point.\

When $r\notin \vert k^\times\vert$, $\widetilde{\mathscr{H}(\eta_{a,r})}=\widetilde{k}$, and $\vert\mathscr{H}(\eta_{a,r})^\times\vert$ is the group generated by $\vert k^\times \vert$ and $r$, so $\eta_{a,r}$ is a type-$3$ point.

\bigskip
\item[•]If $\mathscr{B}=\left(B_n\right)_{n\in \N}$ is a decreasing sequence of non-empty closed balls (i.e. $B_{n+1}\subseteq B_n$, for $n\in \N$) of $k$. Let $\vert\cdot\vert_{B_n}$ be the unique point of $\A_k^{1,\an}$ determined by $B_n$ (i.e. $\vert\cdot\vert_{B_n}=\eta_{a_n, r_n}$ as soon as $B_n=B(a_n,r_n)$). Then the map : 
$$P\in k[T]\mapsto \inf_{n\in \N}\vert P\vert_{B_n}$$
defines an element $\vert \cdot \vert_\mathscr{B}  $ of $\A_k^{1,\an}$.

If $\bigcap_n B_n$ is a point $a\in k$, then $\vert \cdot \vert_\mathscr{B}$ corresponds exactly to $\eta_a$. If $\bigcap_n B_n$ is a closed ball centred in $a\in k$ and of radius $r\in \R_+^*$, then $\vert \cdot \vert_\mathscr{B}$ corresponds to $\eta_{a,r}$. It is also possible that $\bigcap_n B_n$ is empty, in this case $\vert \cdot \vert_\mathscr{B}$ is a type-$4$ point.

\end{itemize}

This description is exhaustive : all the points of $\A_k^{1,\an}$ can be described in this way. 

\begin{rem}
Points of type $4$ exist if and only if $k$ is not \emph{spherically complete}. A valued field is \emph{spherically complete} when it does not admit any \emph{immediate} extension. The field $\C_p$ is not spherically complete, therefore there exist in $\A_{\C_p}^{1,\an}$ some type-$4$ points.
\end{rem}

\bigskip
\textbf{The $k$-analytic projective line} $\P_k^{1, \an}$ is the analytification of the algebraic $k$-projective line $\P_k^1$. It is a proper (compact and without boundary) quasi-smooth connected curve. It admits a rigid point $\infty$ such that there exists a natural isomorphism an $k$-analytic curves : $$\rho : \P_k^{1, \an}\setminus \lbrace \infty \rbrace \iso \A_k^{1, \an}$$

The $k$-analytic affine and projective curves are trees (see \ref{allure graphes reels}), so for each pair $(x,y)$ of points, there exists a unique closed subspace homeomorphic to a compact interval (a segment) with extremities $x$ and $y$. If $a$ et $b$ are in $k$, the segment joining the rigid points $\eta_a$ and $\eta_{b}$ is : $$[\eta_a,\eta_b]=\lbrace \eta_{a,r}\rbrace_{0\leqslant r\leqslant \vert b-a\vert} \cup \lbrace \eta_{b,s}\rbrace_{0\leqslant r\leqslant \vert b-a\vert}.$$
The segment joining $\eta_a$ and $\infty$ is $[\eta_a,\infty]=\lbrace \eta_{a,r}\rbrace_{0\leqslant r\leqslant \infty}$, with $\infty=\eta_{a, \infty}$.

\bigskip
The type of points of $\P_k^{1,\an}$ (or $\A_k^{1,\an})$ can be read on the tree :
\begin{itemize}
\item[•] type-$2$ points are the branching points of the tree,
\item[•] type-$3$ points are the points where nothing special happens (their valence is $2$)
\item[•] type-$1$ or type-$4$ points are the unibranched ones on the tree, there are the "leaves". 
\end{itemize}

\bigskip
\textbf{Analytic skeleton of an analytic curve :} The following notion of \emph{analytic skeleton} of a $k$-analytic curve is the analogous in the analytic word of the dual graph of the special fiber of the stable model of an algebraic $k$-curve. 

\bigskip
A \emph{$k$-analytic disk} is a $k$-analytic curve isomorphic to the analytic domain of $\P_{k}^{1,\an}$ defined by the condition $\vert T\vert \in I$, where $I$ is an interval of the form $[0,r[$, $[0,r]$ for some $r>0$, or $I=[0,+\infty[$.

\begin{defi}[Analytic skeleton]\label{definition analytic skeleton}
The \emph{analytic skeleton} of a quasi-smooth $k$-analytic curve $X$, denoted $S^\an(X)$, is the subset of $X$ consisting of points which do not belong to any open $k$-analytic disk. 
\end{defi}

\begin{prop} [see  \cite{Duc}, $1.6.13$, $5.1.11$] Let $X$ be a quasi-smooth $k$-analytic curve :
\begin{itemize}
\item the analytic skeleton $S^\an(X)$ is a locally finite graph contained in $X_{[2,3]}$, and containing the boundary $\partial^\an X$ of $X$.
\item if $S^\an(X)$ meets all the connected components of $X$, there exists a canonical deformation retraction $r_X : X\to S^\an(X)$. In particular, $X$ and $S^\an(X)$ has the homotopy type.
\end{itemize}
\end{prop}

\begin{rem}
In order to be coherent with the terminology of \cite{M3}, we used in \cite{Gau} the term \emph{semi-graphs} for graphs with potentially "open" edges, i.e. edges which are either not abutting to any vertex, or with only one extremity abutting to one vertex. However we will not make this terminological distinction in this text to avoid some unnecessary heaviness, and speak only \emph{graphs} instead of semi-graphs. 
\end{rem}

\begin{defi}[Truncated skeleton]\label{definition truncated skeleton}
Let $X$ be a quasi-smooth connected $k$-analytic curve with non-empty skeleton $S^\an(X)$, and $r_X : X\to S^\an(X)$ the canonical retraction. The \emph{truncated skeleton} of $X$, denoted $S^\an(X)^\natural$, is the  subgraph of $S^\an(X)$ obtained from $S^\an(X)$ by removing the edges $e$ such that $r_X^{-1}(e)$ is not relatively compact in $X$. 
\end{defi}

\begin{rem}
The edges $e$ of $S^\an(X)$ such that $r_X^{-1}(e)$ is not relatively compact in $X$ are exatly the "open" edges of $S^\an(X)$. So in the terminology of \cite{Gau}, $S^\an(X)^\natural$ is actually the biggest sub-semi-graph of the semi-graph $S^\an(X)$ which is a graph. 
\end{rem}

\begin{defi}[Nodes of the analytic skeleton]
If $x$ is a point of a $k$-analytic curve, its \emph{genus}, denoted $g(x)$, is defined as being $0$ if $x$ is of type $1, 3$ or $4$, and equals the genus of the residual curve (see \cite{Gau}, $3.1.4$) $\mathscr{C}_x$ of $x$ when it is of type $2$. A point $x\in S^\an(X)$ is a \emph{node} of $S^\an(x)$ if it satisfies one of the following conditions : 
\begin{itemize}
\item[•]$x$ is a branching point of $S^\an(X)$ (i.e. $x$ is a vertex of $S^\an(x)$ and at least three different branches of $S^\an(X)$ abut to $x$)
\item[•]$x\in \partial^\an X$
\item[•]$g(x)>0$
\end{itemize}
\end{defi}

\subsection{Analytic annuli : functions, length and torsors}

We are going to define and study the basic properties of $k$-analytic annuli, central in this text.

\bigskip
If $I=[a,b]$ is a compact interval of $\R_{>0}$ (possibly reduced to one point), let $\CC_I$ be the $k$-analytic curve defined as an $k$-affinoid space by :$$\CC_I=\MM\left(k\lbrace b^{-1}T, aU\rbrace / (TU - 1)\right).$$
If $I\subset J$ are compact intervals of $\R_{>0}$, there is a natural morphism $\CC_I\rightarrow \CC_J$ which makes $\CC_I$ a analytic domain of $\CC_J$. If $I$ is an arbitrary interval of $\R_{>0}$, we can define $$\CC_I = \varinjlim_{J\subset I} \CC_J\subset \G_m^\an$$ where $J$ describes all compact intervals of $\R_{>0}$ containing $I$. It would have been equivalent to define $\CC_I$ as the analytic domain of $\P_k^{1, \an}$ defined by the condition $\vert T\vert \in I$. 

\bigskip
\begin{itemize}
\item[•]\textit{Analytic functions :} The $k$-algebra of analytic functions on $\CC_I$ is given by : 

$$\OO_{\CC_I}(\CC_I)=\left\{\sum_{i\in \Z}a_iT^i, a_i\in k, \lim_{\vert i \vert\to +\infty}\vert a_i \vert r^i=0, \forall r\in I   \right\}.$$
\item[•]\textit{Boundary :} If $s<r\in \R_+^*$, $\partial^\an\CC_{\lbrace r \rbrace}=\lbrace \eta_{0,r}\rbrace$, whereas $\partial^\an\CC_{[s,r]}=\lbrace \eta_{0,s}, \eta_{0,r}\rbrace$.
\end{itemize}

\begin{defi}A $k$-analytic \emph{annulus} is defined as a $k$-analytic curve isomorphic to $\CC_I$ for some interval $I$ of $\R_{>0}$. Annuli are quasi-smooth curves.
\end{defi}

\begin{prop}[Condition of invertibility of an analytic function, \cite{Duc} $3.6.6.1$ et $3.6.6.2$]\label{conditions d'inversibilité} Let $I$ be an interval of $\R_{>0}$, and $f=\sum_{i\in\Z}a_iT^i\in \OO_{\CC_I}(\CC_I)$ an analytic function on $\CC_I$. The function $f$ is invertible if and only if there exists an integer $i_0\in \Z$ (necessarily unique) such that $\vert a_{i_0}\vert r^{i_0} >\max_{i\neq i_0} \vert a_i\vert r^i$ for all $r\in I$.

\end{prop}

\begin{rem}\label{remarque conditions d'inversibilité}For an analytic function $f=\sum_{i\in\Z}a_iT^i\in\OO_{\CC_I}(\CC_I)$, we will say that $f$ admits a \emph{strictly dominant monomial} $a_{i_0}T^{i_0}$ when there exists an interger $i_0\in \Z$ such that $\vert a_{i_0}\vert r^{i_0} >\max_{i\neq i_0} \vert a_i\vert r^i$ for all $r\in I$. Such a strictly dominant monomial is unique, and $i_0$ is the \emph{degree} of this monomial. 
The last proposition says that $f\in\OO_{\CC_I}(\CC_I)$ is invertible if and only if it admits a strictly dominant monomial, $a_{i_0}T^{i_0}$, in which case $f$ is written $f=a_{i_0}T^{i_0}(1+u)$ with $u\in \OO_{\CC_I}(\CC_I)$ of norm $<1$ on $\CC_I$.
\end{rem}

\bigskip
Let $f\in \OO_{\CC_I}(\CC_I)^\times$ be an invertible function on $\CC_I$ such that the degree $i_0$ of its strictly dominant monomial is different from $0$. Let $\varphi_f : \CC_I\to \A_k^{1,\an}$ be the morphism induced by $f$, and $\Lambda$ the map from $\R_{>0}$ to itself defined by $r\mapsto \vert a_{i_0} \vert r^{i_0}$.

\begin{prop}[\cite{Duc}, $3.6.8$]
The map $\Lambda$ induces a homeomorphism from $I$ to the interval $\Lambda(I)$ of $\R_{>0}$, and $\varphi_f$ induces a finite and flat morphism of degree $\vert i_0\vert$ from $\CC_I$ to $\CC_{\Lambda(I)}$.
\end{prop}

\begin{defi}[Coordinate functions]
If $\CC$ is a $k$-analytic annulus, une function $f\in \OO_{\CC}(\CC)$ is a \emph{coordinate function} when it induces an isomorphism of $k$-analytic curves $\CC\iso \CC_I$ for some interval $I$ de $\R_{>0}$.
\end{defi}

\begin{coro}[Characterization of coordinate functions, \cite{Duc}, $3.6.11.3$ et $3.6.12.3$]\label{caractérisation des fonctions coordonnées}

An analytic function $f\in \OO_{\CC_I}(\CC_I)$ is a coordinate function if and only if $f$ admits a strictly dominant monomial of degree $i_0\in \lbrace -1, 1\rbrace$. If it is the case, $f$ is invertible in $\OO_{\CC_I}(\CC_I)$ and induces an analytic isomorphism $\CC_I\simeq \CC_{\vert a_{i_0}\vert I^{i_0}}$.

\end{coro}

One can directly deduce the following corollary : 

\begin{coro}\label{caractérisation de couronnes (et des disques) isomorphes à une couronne donnée}
Let $I$ and $I'$ be two intervals of $\R_{>0}$, then $C_{I'}$ is isomorphic to $\CC_I$ if and only if $I'\in \vert k^\times\vert \,I^{\pm 1}.$

\end{coro}\

\begin{rem}[Algebraic characterization of coordinate functions of annuli $(\cite{Duc}, 3.6.13.1)$]\label{remark about annuli}

Define $ \OO_{\CC_I}(\CC_I)^{\circ\circ}$ as the subset of $ \OO_{\CC_I}(\CC_I)$ consisting in functions of norms strictly lower than $1$ on $\CC_I$.
We saw that a function $f\in \OO_{\CC_I}(\CC_I)$ is invertible if and only if it admits a strictly dominant monomial, $a_{i_0}T^{i_0}$. In this case, it can be written $f=a_{i_0}T^{i_0}(1+u)$ with $u\in \OO_{\CC_I}(\CC_I)^{\circ\circ}$, and $\vert f\vert$ equals $\vert a_{i_0}\vert \cdot\vert T\vert ^{i_0}$ on $\CC_I$. Consequently, the group $$\mathscr{Z}_I:= \OO_{\CC_I}(\CC_I)^\times/k^\times\cdot (1+\OO_{\CC_I}(\CC_I)^{\circ\circ})  $$ 
is isomorphic to $\Z$, such an isomorphism is given by the degree of the strictly dominant monomial. From corollary \ref{caractérisation des fonctions coordonnées}, a function $f\in \OO_{\CC_I}(\CC_I)$ is a coordinate function of $\CC_I$ if and only of it is invertible and its class in $\mathscr{Z}_I$ is a generator of $\mathscr{Z}_I$.\\

Therefore, if $\CC$ is any $k$-analytic annulus and $f\in \OO_{\CC}(\CC)$, $f$ is a coordinate function if and only if it is invertible and is sent to a generator of the free abelian group of rank $1$ : $$\mathscr{Z}(\CC):= \OO_{\CC}(\CC)^\times/k^\times\cdot (1+\OO_{\CC}(\CC)^{\circ\circ}).  $$ 
\end{rem}

\begin{defi}[Length of an analytic annulus]\label{length of annuli}
\item If $I$ is an interval of $\R_{>0}$, the length of the annuls $\CC_I$ is defined as : $$\ell(\CC_I)=\log_p\left(\frac{\sup I}{\inf I}\right),$$ with $\ell(\CC_I)=+\infty$ whenever $\inf I=0$ or $\sup I=+\infty$.

The \emph{length} of a general $k$-analytic annulus $\CC$, denoted $\ell(\CC)$, is defined as the length of $\CC_I$ for any interval $I$ of $\R_{>0}$ such that $\CC$ is isomorphic to  $\CC_I$. From corollary \ref{caractérisation de couronnes (et des disques) isomorphes à une couronne donnée} we see that this definition does not depend on the choice of such $I$. 
\end{defi}

There exists a natural distance on the set of type-$2$ and type-$3$ points of $\P_k^{1,\an}$, which is coherent with this definition of length of an annulus. However, we will not define it in this text.

\bigskip

\textbf{\textit{Kummer torsors of an annulus :}} If $X$ is a $k$-analytic space and $\ell\in \N^\times$ and integer (in general it is necessary to ask that $\ell$ is not $0$ in $k$, but it is obviously the case here since we assumed from the beginning that $\mathrm{char}(k)=0$), the \emph{Kummer exact sequence} on $X_{\et}$ :

$$ 1\longrightarrow \mu_\ell \longrightarrow \G_m\overset{z\mapsto z^\ell}{\longrightarrow }\G_m\longrightarrow1 $$

induces an injective morphism $$\mathscr{O}_X(X)^\times / (\mathscr{O}_X(X)^\times)^\ell\overset{\iota}{\hookrightarrow} H^1(X_\et, \mu_\ell)$$ whose image will be denoted $\Kum_\ell(X)$. It is known (\cite{Ber2}) that any locally constant étale sheaf on $X_\et$ is representable. Consequently, $H^1(X_\et, \mu_\ell)$ classyfies all the \emph{analytic} étale $\mu_\ell$-torsors on $X$ up to isomorphism. If $f \in \mathscr{O}_X(X)^\times$, its image $(f)$ in $H^1(X_\et, \mu_\ell)$ by $\iota$ corresponds to $ \mathscr{M}(\mathscr{O}_X[T]/(T^\ell-f)) $. The elements of $\Kum_\ell(X)$ seen as analytic étale $\mu_\ell$-torsors will be called \emph{Kummer $\mu_\ell$-torsors}.

\begin{ex}
If $I$ is a non-empty interval of $\R_{>0}$, the (invertible) function $T^\ell\in \OO_{\CC_I}(\CC_I)^\times$ induces a Kummer $\mu_\ell$-torsor $\mathcal{C}_I\rightarrow \mathcal{C}_{I^\ell}$, identifying $\mathcal{C}_I$ with $\mathscr{M}(\mathscr{O}_{\mathcal{C}_{I^{\ell}}}[T]/(T^\ell-S))$, where $S$ is the standard coordinate of $\mathcal{C}_{I^{\ell}}$.
\end{ex}

\begin{prop}\label{propriétés kummériennes des revêtements de couronnes} Let $\CC$ be a $k$-analytic annulus, and $\ell\in \N^\times$ an integer \emph{prime to the residual characteristic $p$} :
\begin{enumerate}
\item The group $\Kum_\ell(\CC)$ is isomorphic to $\Z/\ell\Z$. This isomorphism is non-canonical as soon as $\ell\geqslant 3$, but becomes canonical when one fixes an orientation of $\CC$.
\item Any connected component of a Kummer $\mu_\ell$-torsor of $\CC$ is an $k$-analytic annulus. 
\item Any $\mu_\ell$-torsor of $\CC$ is Kummer, which leads to an isomorphism : $$ H^1(\mathcal{C}_\et, \mu_\ell) \simeq \Kum_\ell (\CC)\simeq \Z/\ell\Z. $$ 
\end{enumerate}
\end{prop}

\begin{proof}
The proof of the two first points can be found in \cite{Duc}, $3.6.30$ and $3.6.31$. The facts that $k$ is algebraically closed and that $\ell$ is prime to $p$ are necessary for the first point, since it implies that the subgroup  $k^\times\cdot (1+\OO_\CC(\CC)^{\circ\circ})$ of $\OO_\CC(\CC)^\times$ is $\ell$-divisible. It means in terms of the group $\mathscr{Z}(\CC)$ defined in remark \ref{remark about annuli} that :
$\ell\mathscr{Z}(\CC)\simeq \left(\OO_\CC(\CC)^\times\right)^\ell/k^\times\cdot (1+\OO_\CC(\CC)^{\circ\circ})$. Therefore, there is a canonical isomorphism : 
$$\mathrm{Kum}_{\ell}(\CC)\simeq \OO_\CC(\CC)^\times/\left(\OO_\CC(\CC)^\times\right)^\ell\simeq \mathscr{Z}(\CC)/\ell \mathscr{Z}(\CC)\simeq \Z/\ell\Z.$$
The proof of the last point comes from \cite{Ber2}, $6.3.5$, where Berkovich shows that any connected tame finite étale covering of a compact annulus is Kummer, and that any $\mu_\ell$-torsor is tame since $\ell$ is assumed to be prime to $p$. It easy to extend it to the case when $\CC$ is not compact, since it is then identified with the colimit of its compact subannuli.
\end{proof}

\subsection{Tempered fundamental group}

Let $X$ be a quasi-smooth strictly $k$-analytic space (not necessarily a curve). As defined in \cite{DJg}, an \emph{étale covering} of $X$ is a morphism $\varphi : Y\to X$ such that $X$ admits an open covering $X=\bigcup_{i\in I}U_i$ such that for each $i\in I$ : 
$$\varphi^{-1}(U_i)=\coprod_{j\in J_i} Y_{i,j},$$
where each $Y_{i,j}\to U_i$ is finite étale, with potentially infinite index sets. If $X$ is connected, an étale covering $\varphi : Y\to X$ is \emph{Galois} when $Y$ is connected and the action of the automorphism group $G=\Aut(\varphi)$ is simply transitive.  

\bigskip
For instance, finite étale coverings as well as topological coverings (for the Berkovich topology) are étale coverings. Yves André defined in \cite{And} the notion of \emph{tempered covering}, defined as follows : 
\begin{defi}
An étale covering $\varphi : Y\to X$ is \emph{tempered} if it is the quotient of the composition of a topological covering and of a finite étale covering, i.e. if there exists a commutative diagram of $k$-analytic spaces :
\begin{equation*}\label{diagramme commutatif tempéré}
 \xymatrix{ & Z \ar[ld]_\psi \ar[rd]&\\ W \ar[rd]_\chi & & Y \ar[ld]^\varphi\\ & X &}
\end{equation*} where $\chi$ is a finite étale covering and $\psi$ a topological covering. It is equivalent to say that $\varphi$ becomes topological after pullback by some finite étale covering. Let $\Cov^\tp(X)$ be the category of tempered coverings of $X$.
\end{defi}

If $x\in X$ is a geometric point, consider the fibre functor $$F_x : \Cov^\tp(X)\to \mathrm{Set}$$ which maps a covering $Y\to X$ to the fibre $Y_x$. The \emph{tempered fundamental group pointed at $x$} is defined as the automorphism group of the fibre functor in $x$ : 
$$\pi_1^\tp(X,x):=\Aut(F_x).$$ 
The group $\pi_1^\tp(X,x)$ becomes a topological group, by considering the basis of open subgroups consisting of the stabilizers $(\mathrm{Stab}_{F_x(Y)}(y))_{Y\in \Cov^\tp(X),\; y\in F_x(Y)}$. It is a prodiscrete topological group. 

\bigskip
If $x$ and $x'$ are two different geometric points, the functors $F_x$ and $F_{x'}$ are (non canonically) isomorphic, and any automorphism of $F_x$ induces an inner automorphism of $\pi_1^\tp(X,x)$. Thus, one can consider the \emph{tempered fundamental group} $\pi_1^\tp(X)$, defined up to unique outer isomorphism.

\bigskip
If $\pi_1^\mathrm{alg}(X,x)$ (resp. $\pi_1^\mathrm{top}(X,x)$) denotes the group classifying pointed finite étale (resp. topological) coverings of $X$, the natural morphism $\pi_1^\mathrm{temp}(X,x)\to \pi_1^\mathrm{top}(X,x)$ is always surjective, and the natural morphism $\pi_1^\mathrm{temp}(X,x)\to \pi_1^\mathrm{alg}(X,x)$ has dense image, such that $\pi_1^\mathrm{alg}(X,x)$ can be identified with the profinite completion of $\pi_1^\tp(X,x)$ : $$\pi_1^\mathrm{alg}(X,x)=\widehat{\pi_1^\tp(X,x)}.$$ In dimension $1$, when $X$ is a $k$-analytic curve, the morphism $\pi_1^\mathrm{temp}(X,x)\to \pi_1^\mathrm{alg}(X,x)$ is injective (these results can be found in \cite{And}, $2.1.6$). As a consequence, the affine and projective lines $\A_k^{1,\an}$ and $\P_k^{1,\an}$ do not admit any non-trivial tempered coverings : $$\pi_1^\tp(\P_k^{1,\an}) \simeq \pi_1^\tp(\A_k^{1,\an}) \simeq 0.$$

\begin{defi}[Moderate tempered coverings]\label{p' version du groupe fondamental}
Let $\Cov^{\tp, \pp}(X)$ be the full subcategory of $\Cov^\tp(X)$ consisting of tempered coverings which are quotients of a topological covering and a Galois finite étale covering of \emph{degree prime to $p$}. In the same way than for the tempered case, it is possible to consider a classifying group defined as the automorphism group of a geometric fibre functor, and well defined up to unique outer automorphism : this group $\pi_1^{\tp, \pp}(X)$ is called the \emph{moderate tempered group} of $X$. It is naturally a topological pro-discrete group.\end{defi}

\begin{rem}\label{pi1t}
When $X$ is a $k$-analytic curve, the group $\pi_1^{\tp, \pp}(X)$ can be constructed group-theoretically from $\pi_1^\tp(X)$ as the projective limit of quotients of $\pi_1^\tp(X)$ admitting a torsion-free normal subgroup of finite index prime to $p$. 
\end{rem}

\subsection{Verticial, vicinal and cuspidal subgroups}

We recall here some notions and terminology of \cite{Gau} about $k$-analytically hyperbolic curves. Let $X$ be a quasi-smooth connected $k$-analytic curve with non-empty skeleton $S^\an(X)$, and $r_X : X\twoheadrightarrow S^\an(X)$ be the canonical retraction. Let $\Sigma_X$ be the set of vertices of $S^\an(X)$ (it is the set of \emph{nodes} of $S^\an(X)$ in the langage of \cite{Gau}).

\bigskip
An edge $e$ of $S^\an(X)$ can be of two different types :  
\begin{itemize}
\item[•]it is a \emph{vicinal} edge whenever the connected component of $X\setminus \Sigma_X$ associated to $e$, i.e. $r_X^{-1}(\overset{\circ}{e})$, is relatively compact in $X$, which is the same than asking that each of the two extremities of $e$ abuts to a vertex (it is a "closed" edge).

\item[•]it is a \emph{cusp} whenever the associated connected component of $X\setminus \Sigma_X$ is non-relatively compact in $X$, in other words when it is either a isolated edge, or only one of its extremities abuts to a vertex.
\end{itemize}

\begin{rem}
The connected component of $X\setminus \Sigma_X$ associated to a vicinal edge is always a $k$-analytic annulus. However, it might not always be the case for cusps (see \cite{Gau}, Remark $2.18$). A cusp whose associated connected component of $X\setminus \Sigma_X$ is an annulus will be called \emph{annular}.
\end{rem}

Recall from \cite{Gau}, that an étale covering $\varphi : Y\to X$ of a quasi-smooth connected curve $X$ is called \emph{moderate} if for any $y\in Y$, the degree $[\HC(y)^\mathrm{gal}:\HC(\varphi(y))]$ is prime to $p$, where $\HC(y)^\mathrm{gal}$ stands for the Galois closure of the extension $\HC(y)/\HC(\varphi(y))$. The category of moderate covering of $X$ is a Galois category whose fundamental group is denoted $\pi_1^\mathrm{t}(X)$ : the moderate fundamental group of $X$, which is a profinite group. 

\bigskip
Let $e$ be an edge of $S^\an(X)$ and $\CC_e$ the associated connected component of $X\setminus \Sigma_X$. Let $\pi_e=\pi_1^\mathrm{t}(\CC_e)$ the \emph{moderate} fundamental group of $\CC_e$. If $v$ is a vertex of $S^\an(X)$, the \emph{star} centred in $v$, denoted $\mathrm{St}(v,X)$, is defined by $$\mathrm{St(v,X)}=\lbrace v\rbrace \sqcup \bigsqcup_e \CC_e,$$where the disjoint union is taken over all edges $e$ of $S^\an(X)$ abutting to $v$.\
Let ${\pi_v=\pi_1^\mathrm{t}(\mathrm{St}(v,X))}$ be the \emph{moderate} fundamental group of $\mathrm{St(v,X)}$.

\bigskip
We saw in \cite{Gau} that if $X$ is $k$-analytically hyperbolic, for any component $c$ of $S^\an(X)$ (vertex or edge), there is a natural embedding $\pi_c\hookrightarrow \pi_1^{\tp,\pp}(X)$. This comes from the fact that the semi-graph of anabelioids $\GG(X,\Sigma_X)$ is of \emph{injective type} and that there is a natural isomorphism $ \pi_1^\tp(\GG(X,\Sigma_X))\simeq \pi_1^{\tp,\pp}(X) $ (see \cite{Gau}, Corollary $3.36$).

\begin{defi}
If $X$ is a $k$-analytically hyperbolic curve, a compact subgroup of $\pi_1^{\tp,\pp}(X)$ is called : 
\begin{itemize}
\item[•]\emph{vicinal} if it is of the form $\pi_e$ for some vicinal edge $e$ of $S^\an(X)$,
\item[•]\emph{cuspidal} if it is of the form $\pi_e$ for some cusp $e$ of $S^\an(X)$,
\item[•]\emph{verticial} if it is of the form $\pi_v$ for some vertex $v$ of $S^\an(X)$.
\end{itemize}
\end{defi} 

\begin{rem}
The Kummer nature of moderate coverings of an annulus implies that vicinal subgroups as well as cuspidal subgroups are always isomorphic to $\widehat{\Z}^{\pp}$, even for non-annular cusps. However, a compact subgroup of $\pi_1^{\tp,\pp}(X)$ cannot be at the same time vicinal and cuspidal. Verticial subgroups are always isomorphic to the \emph{prime-to-$p$}-profinite completion of the fundamental group of a hyperbolic Riemann surface (see \cite{Gau}, Corollary $3.23$ and proof of $3.30$).
\end{rem}

For a $k$-analytically hyperbolic curve $X$, verticial and vicinal subgroups can be characterised directly from the group $\pi_1^{\tp, \pp}(X)$ : verticial subgroups correspond exactly to (conjugacy classes of) maximal compact subgroup, whereas vicinal subgroups correspond to (conjugacy classes of) non-trivial intersections of two maximal compact subgroups. Therefore one can reconstruct the truncated skeleton $S^\an(X)^\natural$ from the tempered group $ \pi_1^{\tp, \pp}(X)$ (so also from $\pi_1^{\tp}(X)$, since the first one can be deduced from the second taking a suited inverse limit, see \ref{pi1t}).

\section{Harmonic cochains and torsors}\label{section cochaînes harmoniques et torseurs}

\subsection{Splitting conditions of $\mu_p$-torsors}

\begin{lem}\label{lemme sur la distance entre deux racines de l'unité} 
Let $\xi$ and $\xi'$ be two distinct $p^{\mathrm{th}}$-roots of unity in $k$ (recall that $k$ is assumed to be algebraically closed). Then $\vert \xi - \xi'\vert= p^{-\frac{1}{p-1}}$.
\end{lem}

\begin{proof}
Write $\displaystyle{\Phi_p=\frac{X^p-1}{X-1}=\sum_{i=0}^{p-1}X^i=\prod_{\xi\in \mu'_p}X-\xi\in \Q[X]}$ for the $p^{\mathrm{th}}$ cyclotomic polynomial, where $\mu_p'$ stands for the set of the $p-1$ primitive $p^{\mathrm{th}}$-roots of unity in $k$. The evaluation at $1$ gives : $p=\prod_{\xi\in \mu'_p}1-\xi.$ For $\xi$ describing $\mu'_p$, all the $1-\xi$ have the same norm since they are on the same $\Gal(k/\Q)$-conjugacy class. Therefore, $\vert 1-\xi \vert=p^{-\frac{1}{p-1}}$, and we obtain the result since multiplication by any $p^{\mathrm{th}}$-root of unity is an isometry of $k$. 
\end{proof}

An étale coverings $\varphi : Y\to X$ between two $k$-analytic curves \emph{totally splits} over a point $x\in X$ if for any $y\in \varphi^{-1}(\lbrace x\rbrace)$, the extension $\mathscr{H}(x)\to \mathscr{H}(y)$ is an isomorphism. When $\varphi$ is of degree $n$, $\varphi$ totally splits over $x$ if and if the fibre $\varphi^{-1}(\lbrace x\rbrace)$ has exactly $n$ elements, which is the same as saying that locally, over a neighbourhood of $x$, $\varphi$ is a topological covering (see \cite{And}, III, $1.2.1$).

\bigskip
The following proposition, precising the splitting sets of the $\mu_{p^h}$-torsor given by the function $\sqrt[p^h]{1+T}$, will be of paramount important in this article.

\begin{prop}\label{décomposition du torseur sauvage} 
If $h\in \N^\times$, the étale covering $\G_m^\an \xrightarrow{z\mapsto z^{p^h}} \G_m^\an$ totally splits over a point $\eta_{z_0, r}$ satisfying $r<\vert z_0\vert=:\alpha$ if and only if : $r<\alpha p^{-h-\frac{p}{p-1}}$. More precisely, the inverse image of $\eta_{z_0, r}$ contains :
\begin{itemize}
\item[•]only one element when $r\in [\alpha p^{-\frac{p}{p-1}}, \alpha];$
\item[•]$p^i$ elements when $r\in[\alpha p^{-i-\frac{p}{p-1}}, \alpha p^{-i-\frac{1}{p-1}}[,$ with $1\leqslant i\leqslant h-1$;
\item[•]$p^h$ elements when $r\in[0, \alpha p^{-h-\frac{1}{p-1}}[.$
\end{itemize}
\end{prop}

\begin{proof}
Let $f: \G_m^\an\rightarrow \G_m^\an$ be the covering given by $f(z)=z^{p}$. Let $z_1\in k^*$ and $\rho\in \R_+$ satisfying $\rho<\vert z_1 \vert$ (such that $\eta_{z_1, \rho} \notin ]0, \infty[$).\
In order to compute $f(\eta_{z_1,\rho})$, notice that for any polynomial $P\in k[T]$ :
$$\vert P\left(f(\eta_{z_1, \rho})\right) \vert = \vert (P\circ f)(\eta_{z_1, \rho})\vert= \vert P(T^p)(\eta_{z_1, \rho})\vert.$$
Thus, $\vert (T-z_1^p)(f(\eta_{z_1, \rho}))\vert=\vert (T^p-z_1^p)(\eta_{z_1, \rho})\vert$. Moreover : 
$$T^p-z_1^p=\sum_{i=1}^p \binom{p}{i}z_1^{p-i}(T-z_1)^i=\sum_{i=1}^p \gamma_i (T-z_1)^i ,$$
where $\gamma_i= \binom{p}{i}z_1^{p-i}$, with : 
$$
\vert \gamma_i \vert = \left\{
    \begin{array}{ll}
        1& \hbox{si } i=p \\
        p^{-1}\vert z_1 \vert ^{p-i} & \hbox{si } 1\leqslant i\leqslant p-1
    \end{array}
\right.
$$
Consequently, $\vert (T-z_1^p)(f(\eta_{z_1, \rho}))\vert=\max_{1\leqslant i \leqslant p} \lbrace\vert \gamma_i \vert \rho^i \rbrace=\max \lbrace \rho^p, \left( p^{-1}\rho^i \vert z_1 \vert^{p-i}\right)_{1 \leqslant i \leqslant p-1}  \rbrace $. Since we assumed $\rho< \vert z_1 \vert$, we get $\vert (T-z_1^p)(f(\eta_{z_1, \rho}))\vert= \max \lbrace \rho^p, p^{-1}\rho\vert z_1 \vert^{p-1}\rbrace $, that is to say : 
$$
\vert (T-z_1^p)(f(\eta_{z_1, \rho}))\vert = \left\{
    \begin{array}{ll}
         p^{-1}\rho \vert z_1 \vert^{p-1}& \hbox{si } \rho \leqslant \vert z_1 \vert p^{-\frac{1}{p-1}} \\
        \rho^p & \hbox{si } \rho \geqslant \vert z_1 \vert p^{-\frac{1}{p-1}}
    \end{array}
\right.
$$
Define $\widehat{\rho}:=\vert (T-z_1^p)(f(\eta_{z_1, \rho}))\vert $. As $f(\eta_{z_1, \rho})$ is a multiplicative seminorm, for any $k\in \N$ we have  
$\vert (T-z_1^p)^k(f(\eta_{z_1, \rho}))\vert=\widehat{\rho}^k$. By writing : 
$$(T^p-z_1^p)^k=\sum_{j=0}^{pk}\lambda_{k,j}(T-z_1)^j,$$
we obtain : $\widehat{\rho}^k= \max_{0\leqslant j\leqslant pk}\lbrace \vert\lambda_{k,j}\vert\rho^j\rbrace$.\

Let $P=\sum_{k=0}^n\alpha_k(T-z_1^p)^k\in k[X]$ be a polynomial. Up to defining $\lambda_{k, j}:=0$ if $j>pk$,
one can write : \begin{align*}
P(T^p)&=\sum_{k=0}^n \alpha_k(T^p-z_1^p)^k=\sum_{k=0}^n \alpha_k \sum_{j=0}^{pn}\lambda_{k,j}(T-z_1)^j\\
&=\sum_{j=0}^{pn}\underbrace{ \left(\sum_{k=0}^n\alpha_k\lambda_{k,j}\right)}_{:=\widetilde{\alpha}_j}(T-z_1)^j.
\end{align*}
Then we have : 
\begin{align*}
\vert P(f(\eta_{z_1,\rho}))\vert&= \max_{0\leqslant j\leqslant pn}\lbrace \vert \widetilde{\alpha}_j \vert \rho^j\rbrace=\max_{0\leqslant k\leqslant n} \left\lbrace  \vert \alpha_k\vert \cdot\max_{0\leqslant j\leqslant pn}\lbrace \vert\lambda_{k,j}\vert \rho^j\rbrace \right\rbrace\\
&=\max_{0\leqslant k\leqslant n}\lbrace \vert \alpha_k\vert \widehat{\rho}^k \rbrace
\end{align*}
Therefore $f(\eta_{z_1, \rho})=\eta_{z_1^p, \widehat{\rho}}$, which can be written : 

$$
f(\eta_{z_1, \rho})= \left\{
    \begin{array}{ll}
         \eta_{z_1^p,p^{-1}\rho \vert z_1 \vert^{p-1}}& \hbox{si } \rho \leqslant \vert z_1 \vert p^{-\frac{1}{p-1}} \\
        \eta_{z_1^p, \rho^p} & \hbox{si } \rho \geqslant \vert z_1 \vert p^{-\frac{1}{p-1}}
    \end{array}
\right.
$$

\item Let's try to find the preimages by $f$ of $\eta_{z_0, r}$, where $0\leqslant r<\alpha:=\vert z_0 \vert$. Define : 

$$
\widetilde{r} = \left\{
    \begin{array}{ll}
         rp \alpha^{-\frac{p-1}{p}}& \hbox{si } r \leqslant \alpha p^{-\frac{p}{p-1}} \\
        r^{\frac{1}{p}} & \hbox{si } r \geqslant \alpha p^{-\frac{p}{p-1}} 
    \end{array}
\right.
$$
From what is above, if $\widetilde{z_0}$ is a $p^{\mathrm{th}}$-root of $z_0$, then : $$\eta_{\widetilde{z_0}, \widetilde{r}}\in f^{-1}\left(\lbrace \eta_{z_0,r}\rbrace \right),$$ et $f^{-1}\left(\lbrace \eta_{z_0,r}\rbrace \right)$ consists of all conjugates $\eta_{\xi\widetilde{z_0}, \widetilde{r}}$ of $\eta_{\widetilde{z_0}, \widetilde{r}}$ for $\xi\in \mu_p$. Therefore : 
$$
f^{-1}\left(\lbrace \eta_{z_0,r}\rbrace \right) = \left\{
    \begin{array}{ll}
         \lbrace \eta_{\xi \widetilde{z_0},rp\alpha^{-\frac{p-1}{p}}}\rbrace_{\xi\in \mu_p} & \hbox{si } r \leqslant \alpha p^{-\frac{p}{p-1}} \\
         \lbrace \eta_{\xi \widetilde{z_0},r^{\frac{1}{p}}}\rbrace_{\xi\in \mu_p} & \hbox{si } r \geqslant \alpha p^{-\frac{p}{p-1}} 
    \end{array}
\right.
$$
Since $\vert \widetilde{z_0} \vert= \alpha^{\frac{1}{p}}$, we have $\vert \xi \widetilde{z_0}- \xi' \widetilde{z_0}\vert=\alpha^{\frac{1}{p}} p^{-\frac{1}{p-1}}$ as soon as $\xi\neq \xi' \in \mu_p$, from lemma \ref{lemme sur la distance entre deux racines de l'unité}. Thus, $f^{-1}\left(\lbrace \eta_{z_0,r}\rbrace \right)$ has a unique element if $r \geqslant \alpha p^{-\frac{p}{p-1}}$, $p$ otherwise. 
\item For the general case, with $h\geqslant 1$, a recursive reasoning on $h$ leads to the conclusion.

\end{proof}
\begin{figure}
\includegraphics[width=1\textwidth]{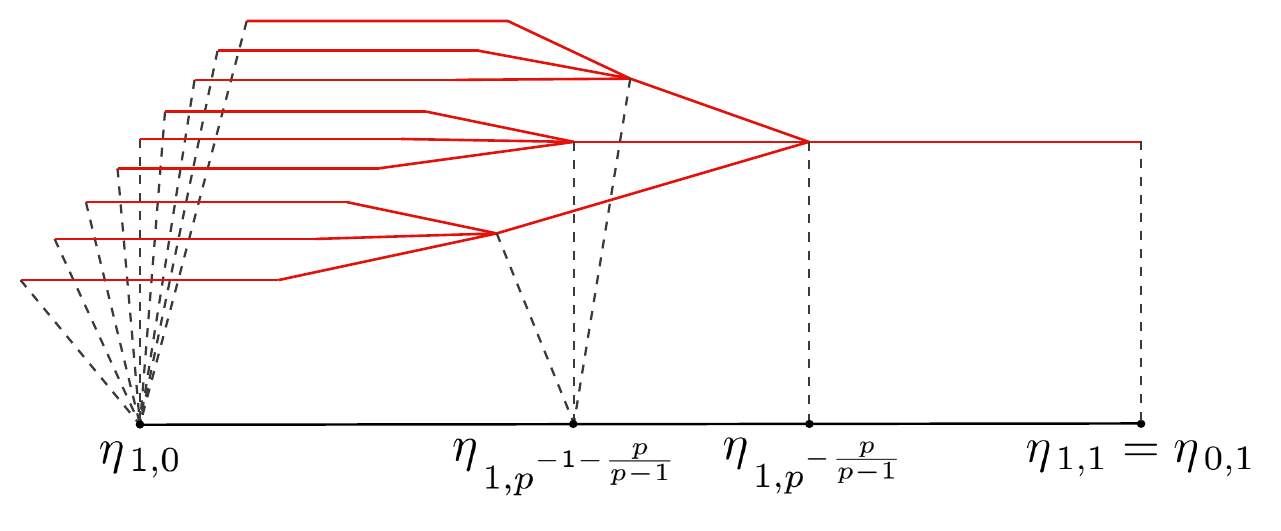}
\caption{Covering $\G_m^\an \xrightarrow{z\mapsto z^{9}} \G_m^\an$ with $p=3$, $h=2$ and $z_0=1$.}
\end{figure}

\subsection{Cochain morphism}\label{sous-section morphisme de cochaîne}

We are going to define the important notion of \emph{$\Z/n\Z$-cochain} associated to a $\mu_n$-torsor. This is exactly from a close look at the behaviours of such cochains of torsors that it will be possible, in section \ref{section anabélianité partielle de la longueur des couronnes}, to extract some information about lengths of annuli.

\begin{defi}[Harmonic cochains]\label{cochaînes harmoniques} 
Let $\Gamma$ be a locally finite graph, and $A$ an abelian group. A \emph{harmonic $A$-cochain on $\Gamma$} is map  $c : \lbrace \text{oriented edges of } \Gamma\rbrace \rightarrow A$ satisfying the two following conditions : 
\begin{enumerate}
\item if $e$ and $e'$ correspond to the same edge with its two different orientations : $c(e')=-c(e)$. 
\item if $x$ is a vertex of $\Gamma$ : $$\sum_{\text{edges oriented towards }x}c(e)=0_A. $$
\end{enumerate}
The set of harmonic $A$-cochains of $\Gamma$ forms an abelian group denoted $\mathrm{Harm}(\Gamma, A)$. In the following, we will simply write $A$-cochains, or cochains when $A$ is explicit.

\end{defi}

Let $X$ be a non-empty $k$-analytic curve with skeleton $\G=S^\an(X)$, and truncated skeleton $\G^\natural=S^\an(X)^\natural$.

\begin{lem}\label{existence du morphisme de cochaîne}
Let $n\in \N^\times$ :
\begin{itemize}
\item  there exists a natural morphism : $$H^1(X,\mu_n)\xrightarrow{\theta} \mathrm{Harm}(\G,\Z/n\Z),$$

\item in the case where $X$ has finite skeleton, does not have any point of genus $>0$, is without boundary and has only annular cusps, then the image of $\theta$ contains $\mathrm{Harm}(\G^\natural,\Z/n\Z)$ (seen as a subgroup of $\mathrm{Harm}(\G,\Z/n\Z)$ by prolongation of cochains by $0$ on all cuspidal edges of $\G$).
\end{itemize}

\end{lem}

\begin{proof}
The exact Kummer sequence gives the following exact sequence :
$$1\rightarrow \mathscr{O}_X(X)^\times/\left(\mathscr{O}_X(X)^\times \right)^n\rightarrow H^1(X, \mu_n)\rightarrow \,_nH^1(X, \G_m)\rightarrow 1,$$
where $_nH^1(X, \G_m)$ denotes the $n$-torsion subgroup of $H^1(X, \G_m)$.\

Moreover $H^1(X, \G_m)=H^1_{\mathrm{top}}(X, \G_m)$ : any étale $\G_m$-torsor is topological, this comes from \cite{Ber2} ($4.1.10$).\ 

Let $h\in H^1(X, \mu_n)$ and $\overline{h}$ its image in $_nH^1(X,\G_m)$. Thus, if $x\in X$, there exists an open neighbourhood $\mathscr{U}$ of $x$ in $X$ such that $\overline{h}$ is trivial on $\mathscr{U}$. Then $h_{\vert \mathscr{U}}$ comes from a function $f\in \mathscr{O}_{\mathscr{U}}(\mathscr{U})^\times$ defined modulo $nth$ powers. There is a natural morphism : $\ds{ \mathscr{O}_{\mathscr{U}}(\mathscr{U})^\times\xrightarrow{\theta_{\mathscr{U}}} \mathrm{Harm}(S^\an(\mathscr{U}), \Z)}$ which factorises through : $$\ds{ \mathscr{O}_{\mathscr{U}}(\mathscr{U})^\times/\left( \mathscr{O}_{\mathscr{U}}(\mathscr{U})^\times\right)^n \rightarrow \mathrm{Harm}\left(S^\an(\mathscr{U}), \Z/n\Z\right)}.$$ This morphism $\theta_{\mathscr{U}}$ can be constructed in the following way : if $e$ is an oriented edge of $S^\an(\mathscr{U})$ and $r$ is the canonical retraction of $\mathscr{U}$ on its skeleton, then $r^{-1}(e)$ is isomorphic to some open annulus of $\P_k^{1,\an}$ defined by the condition $\lbrace 1<\vert T\vert < \lambda \rbrace$, where the beginning of the edge corresponds to $1$, whereas the end corresponds to $\lambda$.\

Let $\tau : \lbrace z\in \P_k^{1,\an},1<\vert T(z)\vert < \lambda \rbrace \iso r^{-1}(e)$ be such an isomorphism, and $\alpha\in \mathscr{O}_{\mathscr{U}}(\mathscr{U})^\times$. There exists a unique $m\in \Z$ such that $\alpha\circ\tau$ is written $z\mapsto z^m g(z)$ with $g$ of constant norm. This comes from the characterization of invertibility of analytic functions on an annulus, and $m$ is the degree of the unique strictly dominant monomial of $\alpha\circ \tau$. It is enough to write $\theta_{\mathscr{U}}(\alpha)(e)=m$, this defines an element of $\mathrm{Harm}(S^\an(\mathscr{U}), \Z)$.\\

We have $S^\an(X)\cap \mathscr{U}\subseteq S^\an(\mathscr{U})$, but the inclusion can be \textit{a priori} strict. However, we are going to show that the support of $\theta_\mathscr{U}(f)$ (i.e. the set of edges $e$ of $S^\an(\mathscr{U})$ such that $\theta_{\mathscr{U}}(f)(e)\neq 0$) is included in $S^\an(X)\cap \mathscr{U}$. Let $e$ be an oriented edge of $S^\an(\mathscr{U})$ not included in $S^\an(X)$. If $y\in e$, $y$ belongs to an open disk $\mathcal{D}$ of $X$. Then there exists a closed disk $\mathcal{D}_0\varsubsetneq\mathcal{D}$ containing $y$ in its interior. As $\mathcal{D}_0$ is a closed disk, its Picard group $\mathrm{Pic}\left(\mathcal{D}_0\right)$ is trivial. Therefore, the $\mu_n$-torsor $h_{\vert \mathcal{D}_0}$ is given by a function $f_0\in \mathscr{O}_{\mathcal{D}_0}(\mathcal{D}_0)^\times$. Moreover, any invertible function on a closed disk has constant norm, hence the cochain associated to $f_0$ at a neighbourhood of $y$ is trivial. In particular, $\theta_{\mathscr{U}\cap \mathcal{D}_0}(f_0)$ is the trivial cochain on $S^\an(\mathscr{U}\cap \mathcal{D}_0)$. Moreover, all these local construction are compatible between each other : $\theta_{\mathscr{U}\cap \mathcal{D}_0}(f_0)=\theta_{\mathscr{U}\cap \mathcal{D}_0}(f)$. Thus $\theta_{\mathscr{U}}(f)(e)=0$, so the support of $\theta_{\mathscr{U}}(f)$ is included in $S^\an(X)\cap \mathscr{U}$.

These local constructions $x\mapsto \theta_{\mathscr{U}}(f)$ can be glued together to finally give a morphism: $H^1(X,\mu_n)\rightarrow \mathrm{Harm}(\G,\Z/n\Z)$.\\

For the second point, let's first explain how to embed $X$ in the analytification of a Mumford $k$-curve : let $X'$ be a proper $k$-analytic curve obtained from $X$ by prolongation of each cusp by a disk. Then $X'$ is the analytification $\mathscr{X}^\an$ of a Mumford $k$-curve $\mathscr{X}$. Moreover $\G^\natural=S^\an(X')$ : the annular cusps of $X$ do not appear anymore in the skeleton of $X'$ since they are prolonged by disks. \

As $\G^\natural=S^\an(X')$, from \cite{Lep$2$} we have a morphism $\overline{\theta} : H^1(X',\mu_n)\rightarrow \mathrm{Harm}(\G, \Z/n\Z)$ whose image exactly equals $\mathrm{Harm}(\G^\natural, \Z/n\Z)$. If $\iota$ denotes the embedding of $X$ in $X'$, there is a commutative diagram : 
$$\xymatrix{ H^1(X',\mu_n) \ar[rr]^{\iota^*} \ar[rd]_{\overline{\theta}} && H^1(X,\mu_n) \ar[ld]^\theta \\ & \mathrm{Harm}(\G, \Z/n\Z)  }$$
which is enough to conclude that $ \mathrm{Harm}(\G^\natural,\Z/n\Z)\subseteq \im(\theta)$.

\end{proof}

\begin{rem}As the morphism $H^1(X,\mu_n)\xrightarrow{\theta} \mathrm{Harm}(\G,\Z/n\Z)$ exists for any $n$, we will now consider $\theta$ as an map :

$$\theta : \bigsqcup_{n\in \N^\times} H^1(X,\mu_n)\rightarrow \bigsqcup_{n\in \N^\times}\mathrm{Harm}(\G,\Z/n\Z)  $$ which induces for each $n\in \N^\times$ a morphism $H^1(X,\mu_n)\rightarrow \mathrm{Harm}(\G,\Z/n\Z)$. \end{rem}

\subsection{Cochains and minimality of the splitting radius on an annulus}

Let $\mathcal{C}$ be a $k$-analytic annulus of finite length, $\alpha\in k$ and $\eta_{\alpha,0}=\eta_\alpha$ (sometimes simply denoted $\alpha$) the assiocated rigid point of $\P_k^{1,\an}$. We are going to show that $\mu_p$-torsors with non-trivial cochains modulo $p$ on $\CC$ satisfy a minimality condition which enables to make them out from trivial cochain torsors.

\begin{defi}
If $X$ is a $k$-analytic curve and $f\in H^1(X,\mu_n)$, let $\mathscr{D}(f)$ denotes the set of points of $X$ over which the analytic torsor defined by $f$ totally splits. 
\end{defi}

\begin{defi}[Splitting radius of a torsor on a rigid point]

Assume $\CC$ is the subannulus  of $\P_k^{1,\an}$ defined by $\vert T\vert \in I$, where $I$ is an interval of $\R_{>0}$. If $\eta_\alpha\in \mathcal{C}$, for any torsor ${f\in H^1(\mathcal{C}, \mu_{p})}$, let $\varrho_f(\alpha)$ be the \emph{splitting radius of $f$ in $\alpha$}, defined by : $$\varrho_f(\alpha)=\sup\left\{r\in]0,\vert \alpha\vert[, \eta_{\alpha, r}\in \mathscr{D}(f)  \right\}   $$

\end{defi}

The following proposition shows how one can detect with this notion the triviality of the $\Z/p\Z$-cochain $\theta(f)$.

\bigskip
\begin{prop}\label{minimalité des rayons de déploiement}
Fix the rigid point $\eta_\alpha\in \CC$, then  $\varrho_f(\alpha)$ is minimal exactly when the cochain of $f\in  H^1(\mathcal{C}, \mu_{p})$ is non-trivial modulo $p$, i.e. when $f\notin \ker(\theta)$.\

More precisely :
\begin{itemize}
\item[•]$f\notin \ker(\theta)$ if and only if $\ds{\varrho_f(\alpha)=\vert \alpha \vert \,p^{-\frac{p}{p-1}}},$
\item[•]$f\in \ker(\theta)$ if and only if $\ds{\varrho_f(\alpha)>\vert \alpha \vert\, p^{-\frac{p}{p-1}}}$.
\end{itemize}
\end{prop}

\begin{proof}
The exact Kummer sequence gives a morphism $$\ds{\mathscr {O}_{\CC}(\mathcal{C})^\times/(\mathscr {O}_{\CC}(\mathcal{C})^\times)^p \hookrightarrow H^1(\mathcal{C},\mu_p)}$$ which becomes an isomorphism when one restricts to any compact subannulus, because of the triviality of the Picard group of any $k$-affinoid subspace. Up to restricting $\CC$, one can assume $f$ is given by a function $g\in \mathscr{O}_{\CC}(\mathcal{C})^\times$, which means that the associated analytic torsor is defined by :
$\ds{\mathscr{O}_{\CC}(\mathcal{C})[S]/(S^p-g)}$.\

Studying the splitting radius of $f$ along the interval $[\eta_\alpha,\eta_{\alpha, \vert \alpha \vert}]$ amounts to doing the change of coordinate function $t:= T-\alpha$ and studying the convergence of $\sqrt[p\;]{g(t+\alpha)}$.\\

\begin{itemize}
\item[•]
Assume $f\notin \ker(\theta)$.There exists $n\in \Z\setminus \lbrace 0 \rbrace$, prime to $p$, such that $g$ has growth rate $n$, i.e. $n$ is the degree of the strictly dominant monomial : $g(T)=a_nT^n(1+u(T))$, with $\vert u\vert <1$ on $\mathcal{C}$. After normalization ($k$ is algebrically closed), one can assume $a_n=1$.\

The series defining $\sqrt[n\,]{1+T}$ has convergence radius equal to $1$ since $n$ is prime to $p$. Therefore, there exists a function $v(T)$ of norm $<1$ on $\CC$ such that $(1+v)^n= 1+u$, so $g(T)=\left(T(1+v)\right)^n$. As $T(1+v)$ is a coordinate function, we can assume $g(T)=T^n$. Since $n$ is prime to $p$, the two $\mu_p$-torsors given by functions $T^n$ et $T$ have the same sets of total splitting, so we can assume $g(T)=T$. Then the result is given by proposition \ref{décomposition du torseur sauvage}.\\

\item[•]Assume $f\in \ker(\theta)$. This means the degree of the strictly dominant monomial of $g$ (the growth rate) is $0$ modulo $p$ : there exists $m\in\Z$ such that $g(T)=a_0T^{mp}(1+u(T))$, with $\vert u\vert <1$ on $\mathcal{C}$. 

Up to division by $T^{mp}$ (it is the class of $g$ modulo $(\mathscr{O}(\mathcal{C})^\times)^p$ which determines the torsor $f$), we can take $m=0$. Let's write 
$$g(T)=\sum_{k\in \Z}a_k T^k.$$
Thus, for all $r\in I$ and $k\in \Z\setminus \lbrace 0\rbrace$, $\vert a_k\vert r^k< \vert a_0 \vert$. Up to normalization and restriction to a subannulus, we can assume that $a_0=1$ and that the extremities of the interval $I$ (open or closed), are $1-\varepsilon$ and $1$ for some $\varepsilon \in ]0,1[$. In this case, for all $k\in \N^\times$, we have $\vert a_k \vert<1$ and $\vert a_{-k} \vert < (1-\varepsilon)^k$. For all $i\geqslant 0$ and $k\in \Z$, let's write ${{k\choose i}=\frac{k(k-1)\ldots (k-i+1)}{i!}}$. Using the generalised binomial expansion, write :
\begin{align*}
g(t+\alpha)&= \sum_{k\in \Z}a_k (t+\alpha)^k\\
                  &= \sum_{k\in \Z}a_k\left(\sum_{i\geqslant 0} {k\choose i}\alpha^{k-i}t^i \right)\\
                  &=\sum_{i\geqslant 0}\, \underbrace{\left(\sum_{k\in \Z}a_k {k\choose i}\alpha^{k-i}\right)}_{A_i} t^i=\sum_{i\geqslant 0} A_i t^i .
\end{align*} 

We have $\vert \alpha \vert \leqslant 1$ since $\eta_\alpha\in \mathcal{C}$, which implies $\vert A_0 \vert = \vert a_0\vert=1$. Writing $v(t)=\sum_{i\geqslant 1}A_it^i$, proposition \ref{décomposition du torseur sauvage} states that the torsor $f$ splits totally on $\eta_{\alpha, r}$ as soon as $\ds{\vert v(\eta_{0,r}) \vert <p^{-\frac{p}{p-1}}}$. Consequently, $\varrho_f(\alpha)\geqslant r$ if $\vert A_i \vert r^i<p^{-\frac{p}{p-1}}$ for all $i\geqslant 1$. Therefore : 
$$\varrho_f(\alpha)\geqslant \inf_{i\geqslant 1} \left\{ \sqrt[i]{\vert A_i \vert^{-1}}p^{-\frac{1}{i}\left(\frac{p}{p-1}  \right)}   \right\}.  $$ Moreover, for all $k\in \Z\setminus \lbrace 0\rbrace$, $\vert a_k \alpha^k \vert <1$. Then, all $i\geqslant 1$ satisfy : $\vert A_i\vert <\vert \alpha\vert^{-i}$, so : 
$$\sqrt[i]{\vert A_i \vert ^{-1}}p^{-\frac{1}{i}\left(\frac{p}{p-1}\right)}>\vert\alpha\vert\, p^{-\frac{1}{i}\left(\frac{p}{p-1}\right)}.$$ 
We deduce : $$\varrho_f(\alpha)\geqslant \min\left\{\vert A_1\vert^{-1}p^{-\frac{p}{p-1}}, \vert\alpha\vert\, p^{-\frac{1}{2}\left(\frac{p}{p-1}\right)}   \right\}> \vert \alpha \vert \,p^{-\frac{p}{p-1}}.$$
\end{itemize}

\end{proof}

\begin{rem}
If $h>1$, it is not true anymore that $\varrho_f(\alpha)$ is minimal if and only if $f\in  H^1(\mathcal{C}, \mu_{p^h})$ has a non-trivial $\Z/p^h\Z$-cochain, i.e. when $f\notin \ker(\theta)$. It is not difficult to show that if $f$ has a cochain \emph{prime to $p$}, then  : $$\ds{\varrho_f(\alpha)=\vert \alpha \vert \,p^{-h-\frac{1}{p-1}}}.$$ However, if $f'\in H^1(\mathcal{C}, \mu_{p^h})$ is the element corresponding to the function $T^p\in \mathscr{O}_{\CC}(\mathcal{C})^\times$, then its $\Z/p^h\Z$-cochain $\theta(f')$ is non-trivial since it equals $p$, but one can show that its splitting radius on $\alpha$ satisfies : 
 $$\ds{\varrho_{f'}(\alpha)\geqslant\vert \alpha \vert \,p^{1-h-\frac{1}{p-1}}}=p\varrho_f(\alpha)>\varrho_f(\alpha),$$ implying that $\varrho_{f'}(\alpha)$ is not minimal even though $f'\notin \ker(\theta)$. Moreover, if the annulus $\CC$ is for instance given by the condition $\vert T\vert \in ]1-\varepsilon, 1[$ with $\varepsilon>0$, the torsor $g\in H^1(\CC,\mu_{p^h})$ given by the function $1+T\in \OO_{\CC}(\CC)^\times$ has trivial cochain, so belongs to $\ker(\theta)$, but its splitting radius on a rigid point $\alpha=\eta_{\alpha}$ is :
 $$\varrho_g(\alpha)=p^{-h-\frac{1}{p-1}}.$$ Consequently, as soon as $\vert \alpha \vert \in ]\frac{1}{p}, 1[$, we have : $\varrho_g(\alpha)<\varrho_{f'}(\alpha)$.
\end{rem}

\begin{coro}\label{corollaire, détection de la nullité de la cochaîne sur une couronne}
If $f\in  H^1(\mathcal{C}, \mu_{p})$, the triviality of the cochain corresponding to $f$ can be detected set-theoretically from the splitting sets of the different $\mu_{p}$-torsors on $\mathcal{C}$  : 

\begin{itemize}
\item[•]\center{$\ds{f\notin \ker(\theta)\;\Longleftrightarrow \;\mathscr{D}(f)_{[2]}\subseteq \bigcap_{f'\in H^1(\mathcal{C}, \mu_{p})}\mathscr{D}(f')_{[2]}}$,}

\center{\begin{align*}
\;\;\;\;\;\bullet \;\;f\in \ker(\theta)\;&\Longleftrightarrow \;\exists f'\in H^1(\mathcal{C}, \mu_{p}), \mathscr{D}(f)_{[2]} \nsubseteq \mathscr{D}(f')_{[2]}\\
&\Longleftrightarrow \forall f'\in H^1(\mathcal{C}, \mu_{p})\setminus\ker(\theta),  \mathscr{D}(f)_{[2]} \nsubseteq \mathscr{D}(f')_{[2]}.
\end{align*}}
\end{itemize}

\end{coro}
\begin{proof}
It is a direct consequence of \ref{minimalité des rayons de déploiement} coupled with the density of $X_{[2]}$ in $X$.
\end{proof}

\subsection{Characterisation of $\mu_p$-torsors with trivial cochain}

The study led so far, which gives a set-theoretic characterization of $\mu_p$-torsors with trivial $\Z/p\Z$-cochain, only deals with $k$-analytic annuli. In order to extend these considerations, we will need a definition and a few restrictions.

\begin{defi}
An edge $e$ of a graph $\mathbb{H}$ is a \emph{bridge} if an only if the map $\pi_0( \mathbb{H}\setminus \lbrace e\rbrace)\to \pi_0(\mathbb{H})$ is not injective, which happens when the edge $e$ "separates" several connected components of $\mathbb{H}\setminus \lbrace e\rbrace$. The graph is said to be \emph{without bridge} when none of its edges is a bridge.
\end{defi}

\begin{prop}\label{proposition de détection du noyau pour les p cochaines}
Let's come back to the $k$-analytic curve $X$ considered in the second part of lemma \ref{existence du morphisme de cochaîne} : without boundary, of finite skeleton, without  points of genus $>0$, and whose cusps are annular. Assume moreover that $\G=S^\an(X)$ is without bridge and there is never strictly more than one cusp coming from each node.\

If $f\in  H^1(X, \mu_{p})$, then $f\in \ker(\theta)$ if and only if, for any vicinal edge $e$ of $S^\an(X)$ of associated annulus $\mathcal{C}_e$, there exists $f_e\in  H^1(X, \mu_{p})$ such that : $$\left(\mathscr{D}(f)_{[2]}  \setminus \mathscr{D}(f_e)_{[2]}    \right)\cap \mathcal{C}_{e\,[2]}\neq \emptyset.$$ 
\end{prop}

\begin{proof}
The assumption that there is never strictly more than one cusp coming from a node implies that a cochain $c\in \mathrm{Harm}(\G,\Z/p\Z)$ is trivial if and only if it is trivial on all vicinal edges of $\G$.

\bigskip
\begin{itemize}
\item[•] Assume that for any vicinal edge $e$ of $\G$ of corresponding annulus $\mathcal{C}_e$, there exists $f_e\in  H^1(X, \mu_{p})$ such that : $\left(\mathscr{D}(f)_{[2]}  \setminus \mathscr{D}(f_e)_{[2]}    \right)\cap \mathcal{C}_{e\,[2]}\neq \emptyset.$ Let $f_{e\, \vert\, \mathcal{C}_e}$ and $f_{\vert \,\mathcal{C}_e}$ be the restrictions of $f_e$ and $f$ to $\mathcal{C}_e$. Then we have $\mathscr{D}(f_{\vert \,\mathcal{C}_e})_{[2]} \nsubseteq \mathscr{D}(f_{e\, \vert\, \mathcal{C}_e})_{[2]}$. With corollary \ref{corollaire, détection de la nullité de la cochaîne sur une couronne}, it implies $\theta(f)(e)=0$. But it is true for any vicinal edge $e$, so $\theta(f)$ is the trivial cochain.
\bigskip
\item[•] Let $f\in  H^1(X, \mu_{p})$, $e$ a vicinal edge of annulus $\mathcal{C}_e$, and assume $f\in \ker(\theta)$. From \ref{minimalité des rayons de déploiement}, as $\theta(f)(e)=0$, we have $\mathscr{D}(f_{\vert \,\mathcal{C}_e})_{[2]} \nsubseteq \mathscr{D}(g_e)_{[2]}$  for all $g_e\in  H^1(\mathcal{C}_e, \mu_{p})$ of non-trivial cochain.

It remains to show that there exists $f_e\in  H^1(X, \mu_{p})$ with non-trivial cochain at $e$.\

From the assumption, the edge $e$ is not a bridge of $\G$, so is neither a bridge of $\G^\natural$. Thus, the evaluation at $e$ : $$\mathrm{Harm}(\G^\natural,\Z/p\Z)\xrightarrow{\mathrm{ev_e}} \Z/p\Z$$ is non-zero. Let's choose $c_e\in \mathrm{ev}_e^{-1}\left( \Z/p\Z\setminus \lbrace 0\rbrace  \right)$, i.e. such that $c_e(e)\neq 0$. From lemma \ref{existence du morphisme de cochaîne}, the image of $\theta$ contains $\mathrm{Harm}(\G^\natural, \Z/p\Z)$. It is enough to take $f_e\in \theta^{-1}(\lbrace c_e \rbrace )$ : we have $\mathscr{D}(f_{\vert \,\mathcal{C}_e})_{[2]} \nsubseteq \mathscr{D}(f_{e\, \vert\, \mathcal{C}_e})_{[2]}$, which can be written :
$$\left(\mathscr{D}(f)_{[2]}  \setminus \mathscr{D}(f_e)_{[2]}    \right)\cap \mathcal{C}_{e\,[2]}\neq \emptyset.$$

\end{itemize}
\end{proof}
\section{Resolution of non-singularities}\label{section résolution des non-singularités}

In algebraic geometry, resolution of non-singularities consists in knowing whether a hyperbolic curve $\mathscr{X}$ admits a finite cover $\mathscr{Y}$ whose stable reduction has some irreducible components above the smooth locus of the stable (or semi-stable) reduction of $\mathscr{X}$. Such techniques happened to be useful in anabelian geometry, see for instance \cite{PoSt} : if $X_0$ is a geometrically connected hyperbolic curve over a finite extension $K$ of $\Q_p$ such that $X_{0, \overline{\Q}_p}$ satisfies such resolution of non-singularities, then any section of $\pi_1^\mathrm{alg}(X_0)\to \Gal (\overline{\Q}_p/K)$ has its image in a decomposition group of a unique valution point.\\

Lepage shows in \cite{Lep$3$} that any Mumford curve over $\overline{\Q}_p$ satisfies a resolution of non-singularities, and applies this result to the anabelian study of the tempered group of such curves. He shows for instance that if $X_1$ and $X_2$ are two Mumford curves over $\overline{\Q}_p$ whose analytifications have isomorphic tempered fundamental groups, then $X_1^\an$ are $X_2^\an$ are naturally homeomorphic (\cite{Lep$3$}, Theorem $3.9$).

\subsection{Definition and properties of solvable points}\label{partie propriétés des points résolubles}

In the framework of this article, we are going to give a \emph{ad hoc} definition of \emph{solvable point}, and \emph{resolution of non-singularities}, in order to stay in the language of analytic geometry without entering in the considerations of (semi-)stable model.

\begin{defi}[Solvable point]
Let $X$ be a $k$-analytic quasi-smooth curve. We will say that a point $x\in X$ \emph{satisfies the resolution of non-singularities}, equivalently \emph{is solvable}, when there exists a finite étale covering $Y$ of $X$ and a node $y$ of $S^\an(Y)$ above $x$. This amounts to "singularising" $x$ to some node of some finite étale covering of $X$, whence the terminology. The set of \emph{solvable points} is denoted  $X_{\res}$.
\end{defi}

\begin{rem}One always has $X_{\res}\subseteq X_{[2]}$. We will say that $X$ \emph{satisfies resolution of non-singularities} when $X_{\res}= X_{[2]}$. Lepage shows in \cite{Lep$3$} (Theorem $2.6$) that the analytification of any Mumford curve over $\overline{\Q}_p$ satisfies resolution of non-singularities.  
\end{rem}

\begin{defi}
If $f\in H^1(X,\mu_n)$, define $\mathscr{D}(f)_{\res}:=\mathscr{D}(f)\cap X_{\res}$ as the set of solvable points of $X$ over which the analytic torsor defined by $f$ totally splits.
\end{defi}

Resolution of non-singularities has a specific anabelian flavour : from the tempered group $\pi_1^\tp(X)$ it is possible to determine the set of solvable points, as well as the set of solvable points belonging to an annulus defined by a vicinal edge, to the skeleton itself, or to the splitting sets of analytic torsors on $X$.

\bigskip
\textbf{Properties : } If $X$ is a $k$-analytically hyperbolic curve, the tempered fundamental group $\pi_1^\tp(X)$ enables to determine :
\begin{itemize}
 \item[•] \textit{the set $X_{\res}$ of solvable points of $X$,}
 \item[•] \textit{if $e$ is a vicinal edge of annulus $\mathcal{C}_e$,
the set $\CC_e\cap X_{\res}$,}
\item[•]\textit{the set $S^\an(X)_{\res}:=S^\an(X)\cap X_{\res}$ of solvable points belonging to the skeleton, }
 \item[•] \textit{if $f$ is a $\mu_n$-torsor of $X$, the set $\mathscr{D}(f)_{\res}$.}
 \end{itemize}  
 \bigskip
 More precisely :\

\begin{enumerate}
\item The decomposition groups $D_x$ of solvable points of $X$ in $\pi_1^\tp(X)$ correspond exactly to the maximal compact subgroups $D$ of $\pi_1^\tp(X)$ such that there exists a open finite index subgroup $H$ of $\pi_1^\tp(X)$ such that the image of $D\cap H$ by the natural morphism $H\to H^{\pp}$ is non-commutative.

\item Let $e$ be a vicinal edge of $S^\an(X)$, and $\mathcal{C}_e$ the associated annulus. If $D_x$ is a decomposition group of a point $x\in X_{\res}$, then $x\in \mathcal{C}_e$ if and only if the image $D_x^{\pp}$ of $D_x$ by the morphism $\pi_1^\tp(X)\to \pi_1^{\tp, \pp}(X)$ is some open in some vicinal subgroup $\pi_e$ associated to $e$.

\item Let $x\in X_{\res}$ be a solvable point. Let $D_x^{\pp}$ be a decomposition group of $x$ in $\pi_1^{\tp, \pp}(X)$, and $Y$ a finite étale covering such that there exists a node $y$ of $S^\an(Y)$ above $x$, which amounts to considering a open subgroup $H$ of $\pi_1^\tp(X)$ of finite index such that $\pi_y=D_x^{\pp}\cap H^{\pp}$ is non-commutative. Let $\iota$ be the morphism $H^{\pp}\to \pi_1^{\tp, \pp}$, then $\iota(\pi_y)$ is an open subgroup of $D_x^{\pp}$. There are three possibilities :
\begin{itemize}

\bigskip
\item[•]\textit{Case $1$ : } $x\notin S^\an(X)$, it is the case when $\iota(\pi_y)$ is trivial.\

\item[•]\textit{Case $2$ : }$x$ is a vertex of the skeleton, it is the case when $\iota(\pi_y)$ is not commutative. In this case $D_x^{\pp}=\pi_x$ is the only verticial subgroup containing $\iota(\pi_y)$, and from Lemma $3.51$ of \cite{Gau} it is also the commensurator of $\iota(\pi_y)$ in $\pi_1^{\tp,\pp}(X)$.\

\item[•]\textit{Case $3$ : }$x$ belongs to an egde $e$ of $S^\an(X)$, this is the case when $\iota(\pi_y)$ is non-trivial and commutative. In this case $D_x^{\pp}=\pi_e$, and $\pi_e$ is the only vicinal or cuspidal subgroup (according to the nature of the edge $e$) containing $\iota(\pi_y)$, it also equals the commensurator of $\iota(\pi_y)$ in $\pi_1^{\tp,\pp}(X)$.
\end{itemize}
\item Let $f\in H^1(X,\mu_n)$ and $D_x$ a decomposition group in $\pi_1^\tp(X)$ of a point $x\in X_{\res}$. Then the knowledge of $\pi_1^\tp(X),$ of $f$ (considered as morphism from $\pi_1^\tp(X)$ to $\Z/n\Z$) and of $D_x$ is enough to determine whether $x\in \mathscr{D}(f)$. 
\end{enumerate}

\bigskip
The point $(1)$, which appears in \cite{Lep$3$}, is a consequence of (\cite{Lep$4$}, prop. $10$) : if $D$ is a compact subgroup of $\pi_1^\tp(X)$, there exists $x\in X$ and a decomposition subgroup $D_x$ of $x$ in $\pi_1^\tp(X)$ such that $D\subseteq D_x$. Therefore, decomposition subgroups in $\pi_1^\tp(X)$ of points of $X$ are exactly the maximal compact subgroups of $\pi_1^\tp(X)$. The image $D_x^{\pp}$ of $D_x$ by the morphism $\pi_1^\tp(X)\to \pi_1^{\tp,\pp}(X)$ is trivial if $x$ does not belong to $S^\an(X)$, non-trivial and commutative (in fact isomorphic to $\widehat{\Z}^{\pp})$) if $x$ belongs to an edge of $S^\an(X)$, and non commutative if $x$ is a vertex of $S^\an(X)$.\

The point $(2)$ comes from the following fact : if $Y\xrightarrow{f} X$ is a finite étale Galois covering of group $G$ with a node $y$ of $S^\an(Y)$ resolving $x\in X_{\mathrm{r\acute{e}s}}$, there exists a canonical retraction $S^\an(Y)/G\twoheadrightarrow  f^{-1}(S^\an(X))/G\simeq S^\an(X)$, and the sub-graph $f^{-1}(S^\an(X))\subseteq S^\an(Y)$ is such that $f^{-1}(S^\an(X))\cap Y_{\res}$ is determined by the data of $\pi_1^\tp(X)$ and of an open subgroup of finite index $H\subseteq \pi_1^\tp(X)$ defining the covering $f$.\

The point $(3)$ can be interpreted as a consequence of lemmas $3.6$ and $3.8$ of \cite{Lep$3$}.\

For the point $(4)$, one needs to have in mind the following fact : if $Y\xrightarrow{f} X$ is a finite étale Galois covering given by an open subgroup $H\subseteq \pi_1^\tp(X)$, the data of the morphism $\iota : H^{\pp}\to \pi_1^{\tp,\pp}(X)$ enables to know the preimage by $f$ of a fixed node $x\in S^\an(X)$. In particular, when $f$ is a $\mu_n$-torsor, the data of $\iota$ enables to know whether $f$ totally splits over $x$ (cf. \cite{Lep$2$}, prop. $7$). Now if $x\in X_{\res}$, if $Z\xrightarrow{g} X$ is a finite étale Galois covering of group $G$ with a node $z\in S^\an(Z)$ which resolves $x$, and if $f\in H^1(X,\mu_n)$ corresponds to the analytic torsor $Y\to X$, the pull-back $Y\times_X Z\to Z$ inherites a natural action of $\mu_n\times G$. Triviality of $f$ over $x$ can be read on the action of $G\times \mu_n$ over $f^{-1}(x)$, i.e. over the $G$-orbit of $z$.

\subsection{Solvability of skeletons of annuli and "threshold" points}

Let $X$ be a quasi-smooth $k$-analytic curve without boundary, of finite skeleton, without points of genus $>0$, and whose cusps are annular.  We saw that $X$ can be considered as a non-empty open subset of the analytification $X'$ of a Mumford $k$-curve. However, one cannot \textit{a priori} conclude that $X$ satisfies resolution of non-singularities, since it is not defined over $\overline{\Q}_p$ : the proof of theorem $2.6$ of \cite{Lep$3$} does not work well anymore when $\vert k^\times \vert$ is uncountable. Nevertheless, we will only need a very weak version of resolution of non-singularities : it will be enough for us to have the solvability of type-$2$ points of skeletons of annuli, as well as the one of "threshold" branching points of $\mu_p$-torsors of non-trivial cochain.

\begin{lem}\label{lemme de résolubilité des points seuils}
Let $e$ be an edge of $S^\an(X)$ of corresponding annulus $\CC_e$, and $\alpha$ a rigid point of $\CC_e$ which does not belong to the skeleton of $\CC_e$. Let $r$ be the canonical retraction from $X$ to its skeleton, and $x_\alpha\in \CC_e$ the unique point of $]\alpha, r(\alpha)[$ situated to the distance $\frac{p}{p-1}$ of $r(\alpha)$. If $e$ is not a bridge of $\G$, then $x_\alpha\in X_{\res}$.
\end{lem}

\begin{proof}

From proposition \ref{minimalité des rayons de déploiement}, we know that any analytic $\mu_p$-torsor $Y\rightarrow X$ whose cochain is non-trivial modulo $p$ on $e$ is non-split with a unique preimage over each point of $[x_\alpha,r(\alpha)[$, and totally split over $[\alpha, x_\alpha[$. Moreover, such a $\mu_p$-torsor of $X$ with non-trivial cochain on $e$ exists : as in the proof of proposition \ref{proposition de détection du noyau pour les p cochaines}, the fact that $e$ is not a bridge implies that the set $$(\mathrm{ev}_e\circ \theta)^{-1}\left(\left(\Z/p\Z\right)^\times   \right) \subset H^1(X, \mu_p) $$ is non-empty.

If $Y\to X$ is such a torsor, the unique preimage $y\in Y$ of $x_\alpha$ is a branching point, thus a node of $S^\an(Y)$ living above $x_\alpha$.
Therefore, $x_\alpha$ is a solvable point of $X$.
\end{proof}

\begin{rem}
Such a point $x_\alpha$ situated to a distance $\frac{p}{p-1}$ of the skeleton is a "threshold" point. Indeed, travelling along the segment $[\alpha, r(\alpha)]$, $x_\alpha$ is exactly the threshold point until which any $\mu_p$-torsor with non-trivial $\Z_p/\Z$-cochain on $e$ totally splits.
\end{rem}

\begin{lem}\label{lemme de résolubilité du squelette d'une couronne}
Let $\CC$ be a $k$-analytic annulus. Then all the type-$2$ points of the analytic skeleton of $\CC$ are solvable : $\;\;S^\an(X)_{\res}=S^\an(X)_{[2]}.$
\end{lem}

\begin{proof}
One can assume there exists a non-empty interval $I$ of $\R_{>0}$ such that $\CC$ is the analytic domain of $\P_k^{1,\an}$ defined by the condition $\vert T-1\vert \in I$, that we will denote $\CC_{(I)}$. Let $x\in \CC_{[2]}$, there exists $r\in I$ such that $x=\eta_{\,1,r}$. Up to replacing the interval $I$ by $J=r^{-1}p^{-\frac{p}{p-1}}I$ (in this case $\CC_{(I)}\simeq \CC_{(J)}$ because $r^{-1}p^{-\frac{p}{p-1}}\in \vert k^\times\vert$), one can assume $x=\eta_{1,p^{-\frac{p}{p-1}}}$. From proposition \ref{décomposition du torseur sauvage}, the point $x$ admits only one preimage $y$ by the étale covering $Y\to \CC$ induced by $\G_m^\an\xrightarrow{z\mapsto z^p} \G_m^\an$, and $y$ is a branching point of $Y$, i.e. a node of $S^\an(Y)$. Therefore $x$ is a branching point of $X$.
\end{proof}

\subsection{Anabelianity of the triviality of $\mu_p$-torsors}

We are now up to proving some tempered anabelianity of the triviality of cochains associated to $\mu_p$-torsors on a curve $X$ : either when $X$ is an annulus, or a $k$-analytically hyperbolic curve which is some open subset of the analytification of a Mumford $k$-curve.

\begin{thm}\label{théorème d'anabélianité de la trivialité des cochaînes}  Let $X$ be a $k$-analytic curve satisfying one of the two following conditions:
\begin{enumerate}
\item $X$ is an annulus 
\item$X$ is a $k$-analytically hyperbolic curve, of finite skeleton without bridge, without any point of genus $>0$, without boundary, with only annular cusps, and such that there is never strictly more than one cusp coming from each node.

\end{enumerate}

Then the set of $\mu_p$-torsors of $X$ of trivial $\Z/p\Z$-cochain, i.e. the set $H^1(X,\mu_p)\cap \ker(\theta)$, is completely determined by $\pi_1^\tp(X)$.

\end{thm}

\begin{proof}
Let's concentrate on the second case, when the curve is $k$-analytically hyperbolic. The case of an annulus is treated exactly the same, inspiring from corollary \ref{corollaire, détection de la nullité de la cochaîne sur une couronne}, rather than proposition \ref{proposition de détection du noyau pour les p cochaines}.\

From \ref{proposition de détection du noyau pour les p cochaines}, an element $f\in H^1(X,\mu_p)$ belongs to $\ker(\theta)$ if and only if, for any vicinal edge $e$ of $S^\an(X)$ of associated annulus $\mathcal{C}_e$, there exists $f_e\in  H^1(X, \mu_{p})$ such that :  $$\left(\mathscr{D}(f)_{[2]}  \setminus \mathscr{D}(f_e)_{[2]}    \right)\cap \mathcal{C}_{e\,[2]}\neq \emptyset,$$ and one can always choose $f_e$ such that $\theta(f_e)(e)\neq 0$. In this case, as soon as $\alpha$ is a rigid point of $\CC_e$, the "threshold" point $x_\alpha\in ]\alpha,r(\alpha)[$ situated at a distance $\frac{p}{p-1}$ of $r(\alpha)$ is split by $f$ but not by $f_e$ : this comes from proposition \ref{minimalité des rayons de déploiement}. Therefore $x_\alpha\in \mathscr{D}(f)\setminus \mathscr{D}(f_e)$, and as such points are solvable by lemma \ref{lemme de résolubilité des points seuils}, we obtain : $$x_\alpha\in \mathscr{D}(f)_{\res}\setminus \mathscr{D}(f_e)_{\res}.$$

Since $X_{\res}\subseteq X_{[2]}$, we have $f\in \ker(\theta)$ if and only if there exists $f_e\in H^1(X,\mu_p)$ such that
\begin{equation}\label{équation de déploiement}
\left(\mathscr{D}(f)_{\res}  \setminus \mathscr{D}(f_e)_{\res}    \right)\cap \mathcal{C}_e\neq \emptyset.
\end{equation}
From the properties of solvable points presented in \ref{partie propriétés des points résolubles}, the sets $\mathscr{D}(f)_{\res}$, $\mathscr{D}(f_e)_{\res}$ and $\CC_e\cap X_{\res}$ are determined by the tempered group $\pi_1^\tp(X)$, so the condition (\ref{équation de déploiement}) above can be detected from the tempered group, hence the result. 
\end{proof}

\begin{rem}
The second condition on the curve $X$ amounts to asking that $X$ is $k$-analytically hyperbolic curve isomorphic to a open subset of the analytification of Mumford $k$-curve, such that $S^\an(X)$ is without bridge and there is never strictly more than one cusp coming from each node. In this case, from Theorem $3.63$ of \cite{Gau}, $X$ is $k$-analytically anabelian, not only hyperbolic.

\end{rem}

\begin{coro}
Let's stay in the framework of theorem \ref{théorème d'anabélianité de la trivialité des cochaînes}. Let $h\in \N^\times$, and $\mathrm{mod}(p) : \mathrm{Harm}(\G,\Z/p^h\Z)\to \mathrm{Harm}(\G,\Z/p\Z)$ the reduction modulo $p$ of the $\Z/p^h\Z$-cochains. Then it is possible to characterize from the tempered group $\pi_1^\tp(X)$ the kernel of the composed morphism $$\mathrm{mod}(p)\circ \theta : H^1(X,\mu_{p^h}) \to \mathrm{Harm}(\G,\Z/p\Z).$$ 
\end{coro}

\begin{proof}
We have a commutative diagram : 
$$\xymatrix{H^1(X,\mu_{p^h}) \ar[r]^\theta \ar[d]& \mathrm{Harm}(\G,\Z/p^h\Z) \ar[d]^{\mathrm{mod}(p)}\\ H^1(X,\mu_{p}) \ar[r]^\theta &  \mathrm{Harm}(\G,\Z/p\Z) }$$
where the first vertical arrow is induced by the exact sequence :
$\ds{1\to \mu_{p^{h-1}} \to \mu_{p^h} \xrightarrow{\pi} \mu_p \to 1}$. With the identification $H^1(X,\mu_{p^i})\simeq \Hom(\pi_1^\tp(X),\mu_{p^i})$, this morphism is nothing else than $$\Hom(\pi_1^\tp(X),\mu_{p^h}) \xrightarrow{\pi_*} \Hom(\pi_1^\tp(X),\mu_{p}),$$ so only depends on the tempered group $\pi_1^\tp(X)$. The conclusion follows from \ref{théorème d'anabélianité de la trivialité des cochaînes} and the commutativity of the diagram. 
\end{proof}

\section{Partial anabelianity of lengths of annuli}\label{section anabélianité partielle de la longueur des couronnes}

We are going to show how all these set-theoretical considerations about the intersection of the skeleton of an annulus with the splitting sets of its $\mu_p$-torsors enable to extract some information about the length of the annulus, before giving an anabelian interpretation.

\subsection{Lengths and splitting sets}

The following lemma enables to determine whether the length of an annulus is $>\frac{2p}{p-1}$ from the knowledge of its $\mu_p$-torsors of trivial cochain.

\begin{lem}\label{lemme pour montrer qu'une couronne est de longueur supérieure à la borne}
A $k$-analytic annulus $\mathcal{C}$ has a length strictly greater than $\frac{2p}{p-1}$ if and only if any $\mu_p$-torsor of trivial cochain on $\CC$ totally splits over a non-empty portion of its analytic skeleton :
$$\ell(\mathcal{C})>\frac{2p}{p-1}\iff \forall f\in H^1(\mathcal{C},\mu_p)\cap\ker(\theta), \mathscr{D}(f)_{[2]}\cap S^\an(\mathcal{C})\neq \emptyset.$$
\end{lem}

\begin{proof}
Assume $\ell(\mathcal{C})>\frac{2p}{p-1}$, and consider $f\in H^1(\mathcal{C},\mu_p)\cap \ker(\theta)$. As in the proof of proposition \ref{minimalité des rayons de déploiement}, up to restricting $\mathcal{C}$ (but slightly, in order to keep the condition on the length), one can assume that $\mathcal{C}$ is the subannulus of $\P_k^{1,\an}$ given by the condition $T\in ]1-\varepsilon,1[$ with $1-\epsilon<p^{-\frac{2p}{p-1}}$, and that the $\mu_p$-torsor $f$ is defined by a function $g\in \mathscr{O}_\CC(\mathcal{C})^\times$ written : $$g(T)=1+\underbrace{\sum_{k\in \Z\setminus \lbrace 0\rbrace}a_kT^k}_{u(T)},$$
with for all $k\in \N^\times$ : $\vert a_k \vert < 1$ and $\vert a_{-k}\vert<(1-\varepsilon)^k$. The skeleton of $\mathcal{C}$ is the interval $]\eta_{0,1-\varepsilon},\eta_{0,1}[$, and the corresponding analytic torsor totally splits over a point $\eta_{0, r}\in S^\an(\mathcal{C})$ as soon as $\ds{\vert u(\eta_{0, r}) \vert<p^{-\frac{p}{p-1}}}$. Let $k\in \N^\times$ :
\begin{itemize}
\item[•] if $r<p^{-\frac{p}{p-1}}$, $\vert a_k r^k\vert<p^{-\frac{kp}{p-1}}\leqslant p^{-\frac{p}{p-1}}$
\bigskip
\item[•] if $r>(1-\varepsilon)\,p^{\frac{p}{p-1}}$, $\vert a_{-k}r^{-k}\vert< (1-\varepsilon)^k(1-\varepsilon)^{-k}p^{-\frac{kp}{p-1}}=p^{-\frac{kp}{p-1}}<p^{-\frac{p}{p-1}} $
\end{itemize}\

But we have $1-\varepsilon<p^{-\frac{2p}{p-1}}$ (from the assumption on the length of $\mathcal{C}$), hence : $$r_1=(1-\varepsilon)\,p^{\frac{p}{p-1}}<p^{-\frac{p}{p-1}}=r_2.$$ Consequently, the torsor $f$ totally splits over the non-empty interval $]\eta_{0,r_1},\eta_{0,r_2}[$ of the skeleton. From the density of $S^\an(\mathcal{C})_{[2]}$ in $S^\an(\mathcal{C})$, we obtain that $\mathscr{D}(f)_{[2]}\cap S^\an(\mathcal{C})\neq \emptyset$.

\bigskip
Reciprocally, if $\ell(\mathcal{C})\leqslant \frac{2p}{p-1}$, one can check that the torsor given by the function ${g(T)=1+T+(1-\varepsilon)T^{-1}}$ never totally splits over any point of $S^\an(\mathcal{C})$.
\end{proof}

It is actually possible to reduce by half the previous bound from a finer condition requiring to look at the set of $\mu_p$-torsors which totally split \emph{over a neighbourhood of a fixed extremity}. We need  the following definition :
\begin{defi}
Let $\mathcal{C}$ a non-empty $k$-analytic annulus. Its skeleton $S^\an(\mathcal{C})$ is an interval (open or close), and let $\omega$ be one of its extremities. Let $H^1_\omega(\mathcal{C},\mu_p)$ be the subgroup of $H^1(\mathcal{C},\mu_p)$ of $\mu_p$-torsors which totally split \emph{over a neighbourhood of $\omega$}, i.e. which totally splits over a subinterval of $S^\an(\CC)$ of non-empty interior, and whose complementary in $S^\an(\mathcal{C})$ is an interval which does not admit $\omega$ as extremity.

\end{defi}

\begin{lem}\label{lemme pour la meilleure borne}
A $k$-analytic annulus $\mathcal{C}$ has a length strictly greater than $\frac{p}{p-1}$ if and only if, for any extremity $\omega$ of $S^\an(\mathcal{C})$ : 
$$\mathrm{Card} \left( \bigcap_{f\in H^1_\omega(\mathcal{C},\mu_p)} \mathscr{D}(f)_{[2]}\cap S^\an(\mathcal{C}) \right)\geqslant 2.  $$
\end{lem}

\begin{proof}
Assume $\ell(\mathcal{C})>\frac{p}{p-1}$, and consider $f\in H^1_\omega(\mathcal{C},\mu_p)$. Up to restriction of $\mathcal{C}$ (but slightly, such that the condition on the length still holds), one can assume $\mathcal{C}$ is the subannulus of $\P_k^{1,\an}$ given by the condition $T\in ]1-\varepsilon,1]$ with $1-\varepsilon<p^{-\frac{p}{p-1}}$. Let $\mathcal{D}_0$ be the closed $k$-analytic disk of $\P_k^{1,\an}$ centred in $0$ and of radius $1$, i.e. defined by the condition $\vert T\vert \in [0,1]$. The annulus $\mathcal{C}$ is then a $k$-analytic subspace of $\mathcal{D}_0$. From the assumption on $f$ it is possible to extend $f$ into a torsor $\widetilde{f}\in H^1(\mathcal{D}_0, \mu_p)$ of $\mathcal{D}_0$ trivial over $\mathcal{D}_0\setminus \mathcal{C}$. Since $\mathrm{Pic}(\mathcal{D}_0)$ is trivial ($\mathcal{D}_0$ is a $k$-affinoid space), $\widetilde{f}$ is given by a function $g\in \mathscr{O}_{\mathcal{D}_0}(\mathcal{D}_0)^\times$ written $$g(T)=1+\underbrace{\sum_{k\in \N^\times}a_kT^k}_{v(T)},$$
with $\vert a_k \vert < 1$ for all $k\in \N^\times$.\

The skeleton of $\mathcal{C}$ is the interval $]\eta_{0,1-\varepsilon},\eta_{0,1}]$, and the torsor $f=\widetilde{f}_{\vert\, \mathcal{C}}$ totally splits over the point $\eta_{0, r}\in S^\an(\mathcal{C})$ as soon as $\ds{\vert v(\eta_{0, r}) \vert<p^{-\frac{p}{p-1}}}$.\
 
 For all $k\in \N^\times$ and $r\in ]1-\varepsilon, p^{-\frac{p}{p-1}}[$, we have $\vert a_k r^k \vert <p^{-\frac{p}{p-1}}$, so ${\vert v(\eta_{0, r}) \vert<p^{-\frac{p}{p-1}}}$. Thus, the interval $]\eta_{0, 1-\varepsilon}, \eta_{0, p^{-\frac{p}{p-1}}}[$ belongs to $\mathscr{D}(f)$. As the reasoning is independent of the choice of  $f\in H^1_\omega(\mathcal{C},\mu_p)$, one obtain : $$]\eta_{0, 1-\varepsilon}, \eta_{0, p^{-\frac{p}{p-1}}}[ \subseteq \bigcap_{f\in H^1_\omega(\mathcal{C},\mu_p)} \mathscr{D}(f)\cap S^\an(\mathcal{C}),$$ and the conclusion follows from density of type-$2$ points in $\CC$.
 
 \bigskip
Reciprocaly, consider an annulus $\mathcal{C}$ of length $\ell(\mathcal{C})\leqslant \frac{p}{p-1}$, and assume there exist two distinct points $\ds{x_1, x_2 \in\bigcap_{f\in H^1_\omega(\mathcal{C},\mu_p)} \mathscr{D}(f)_{[2]}\cap S^\an(\mathcal{C})}$. Let $y\in ]x_1,x_2[$ a type-$2$ point. Let $I$ be the connected component of $S^\an(C)\setminus \lbrace y \rbrace$ which does not abut to $\omega$, and $\mathcal{C}_I$ the subannulus of $\mathcal{C}$ of skeleton $I$. Up to exanging $x_1$ and $x_2$, one can assume $x_2\in I$. As the annulus $\mathcal{C}_I$ has a length $<\frac{p}{p-1}$, there exists $h\in H^1(\mathcal{C}_I,\mu_p)$ such that $\mathscr{D}(h)\cap S^\an(\mathcal{C}_I)=]y, x_2[$. Therefore, $h$ can be extended into a torsor $\widetilde{h}\in H^1(\mathcal{C}, \mu_p)$ of $\mathcal{C}$, such that $\mathscr{D}(\widetilde{h})\cap S^\an(\mathcal{C})= ]x_1, y[$ (or $[x_1, y[$, according to whether $\mathcal{C}$ is open or closed in $x_1$). Then $\widetilde{h}\in H^1_\omega(\mathcal{C},\mu_p)$, which leads to a contradiction since $x_2\notin \mathscr{D}(\widetilde{h})$.
 
\end{proof}

\begin{coro}\label{corollaire permettant de caractériser si une couronne est plus grande que le seuil}
It is possible to determine from the tempered fundamental group of a $k$-analytic annulus whether the length of the latter is strictly greater than $\frac{p}{p-1}$. 
\end{coro}

\begin{proof}
We showed in lemma \ref{lemme de résolubilité du squelette d'une couronne} that all type-$2$ points of the skeleton of $\CC$ are solvable : $\;S^\an(\CC)_{\res}=S^\an(\CC)_{[2]}.$ Thus : 
$$\bigcap_{f\in H^1_\omega(\mathcal{C},\mu_p)} \mathscr{D}(f)_{[2]}\cap S^\an(\mathcal{C})=\bigcap_{f\in H^1_\omega(\mathcal{C},\mu_p)} \mathscr{D}(f)_{\res}\cap S^\an(\mathcal{C})_{\res}.$$

From the properties of solvable points presented in \ref{partie propriétés des points résolubles}, the tempered group $\pi_1^\tp(\CC)$ characterises the sets $\mathscr{D}(f)_{\res}$ and $S^\an(\mathcal{C})_{\res}$. Moreover, a torsor $f\in H^1(\CC,\mu_p)$ belongs to $H^1_\omega(\CC,\mu_p)$ if and only if it totally splits over the set of type-$2$ points of a non-empty neighbourhood of $\omega$ in $S^\an(\CC)$. But all the type-$2$ points of $S^\an(\CC)$ are solvable, so $H^1_\omega(\CC,\mu_p)$ is itself characterized by the tempered group. The result follows from lemma \ref{lemme pour la meilleure borne}.
\end{proof}

\subsection{Results on lengths of annuli}

We are now in a position to state our result of partial anabelianity of lengths of annuli. Even if we are not yet in a position to know whether the fundamental group of an annulus determines its length, the following result shows that the lengths of two annuli which have isomorphic tempered fundamental groups cannot be too far from each other. When the lengths are finite, we give an explicit bound, depending only on the residual characteristic $p$, for the absolute value of the difference of these lengths.

\begin{thm}\label{théorème sur la longueur des couronnes}
Let $\mathcal{C}_1$ and $\mathcal{C}_2$ be two $k$-analytic annuli whose tempered fundamental groups $\pi_1^\tp(\mathcal{C}_1)$ and $\pi_1^\tp(\mathcal{C}_2)$ are isomorphic. Then $\mathcal{C}_1$ has finite length if and only if $\mathcal{C}_2$ has finite length. In this case : 
$$\vert\ell(\mathcal{C}_1)-\ell(\mathcal{C}_2)\vert < \frac{2p}{p-1}.$$ We also have $d\left( \ell(\mathcal{C}_1),p\N^\times    \right)>1$ if and only if $d\left( \ell(\mathcal{C}_2),p\N^\times    \right)>1$, and in this case : $$\vert\ell(\mathcal{C}_1)-\ell(\mathcal{C}_2)\vert < \frac{p}{p-1}.$$
\end{thm}

\begin{proof}
Let $n\in\N^\times$ prime to $p$, and $i\in \lbrace 1,2 \rbrace$. We know that all $\mu_n$-torsors of $\mathcal{C}_i$ are Kummer. Thus, annuli defined by torsors coming from $H^1(\mathcal{C}_i,\mu_n)$ have length $\frac{\ell(\mathcal{C}_i)}{n}$ (with potentially $\ell(\mathcal{C}_i)=+\infty$).\

Moreover, all the $\mu_n$-torsors of $\mathcal{C}_i$ can be "read" on the tempered group $\pi_1^\tp(\mathcal{C}_i)$ since $H^1(\mathcal{C}_i,\mu_n)\simeq \Hom(\pi_1^\tp(\mathcal{C}_i),\mu_n)$. From corollary \ref{corollaire permettant de caractériser si une couronne est plus grande que le seuil} it is then possible, from $\pi_1^\tp(\mathcal{C}_i)$, to know whether $\frac{\ell(\mathcal{C}_i)}{n}>\frac{p}{p-1}$, and step by step to find the smallest integer $j$ such that : $$\frac{\ell(\mathcal{C}_i)}{n^{j+1}}\leqslant \frac{p}{p-1}<\frac{\ell(\mathcal{C}_i)}{n^j},\;\;\;\;\;\text{i.e. such that }\;\;\;\; n^j\frac{p}{p-1}<\ell(\mathcal{C}_i)\leqslant n^{j+1}\frac{p}{p-1}.$$ But the tempered groups of these two annuli are isomorphic, so such a $j$ will be the same for $\mathcal{C}_1$ and $\mathcal{C}_2$. In particular, for any $N\in \N^\times$ \textit{prime to $p$} : 
$$N\frac{p}{p-1}<\ell(\mathcal{C}_1) \iff N\frac{p}{p-1}<\ell(\mathcal{C}_2),$$
which leads to the conclusion.
\end{proof}


\newpage
\begin{thebibliography}{2}


\bibitem[And]{And} \textsc{Yves André}, \textit{Period Mapping and Differential Equations : From $\C$ to $\C_p$}, MSJ Memoirs, Mathematical Society of Japan \textbf{12} ($2003$).\\

\bibitem[And1]{And1} \textsc{Yves André}, \textit{On a geometric description of $\Gal(\overline{\Q}_p/\Q_p)$ and a $p$-adic avatar of $\widehat{GT}$}, Duke Mathematical Journal \textbf{119} ($2003$), pp. $1-39$.\\

\bibitem[Ber$1$]{Ber1} \textsc{Vladimir Berkovich}, \textit{Spectral Theory and Analytic Geometry Over Non-Archimedean Fields}, Mathematical Surveys and Monographs, American Mathematical Society \textbf{33} ($1990$).\\


\bibitem[Ber$2$]{Ber2} \textsc{Vladimir Berkovich}, \textit{Étale cohomology for non-archimedean analytic spaces}, Publications Mathématiques de l'Institut des Hautes Études Scientifiques \textbf{78} ($1993$), pp. $5-161$.\\

\bibitem[DJg]{DJg} \textsc{Aise De Jong}, \textit{Etale fondamental groups of non-Archimedian analytic spaces}, Compositio Math.  \textbf{97} ($1995$), pp. $89-118$.\\

\bibitem[Duc]{Duc} \textsc{Antoine Ducros}, \textit{La structure des courbes analytiques} (project of book).\\


\bibitem[Gau]{Gau} \textsc{Sylvain Gaulhiac}, \textit{Reconstruction du squelette des courbes analytiques} (preprint) arXiv :$1904.03126$v$2$.\\

\bibitem[Lep1]{Lep$1$} \textsc{Emmanuel Lepage}, \textit{Tempered fundamental group and metric graph of a Mumford curve}, Publications of the Research Institute for Mathematical Sciences \textbf{46} ($2010$), no.$4$, pp. $849-897$.\\

\bibitem[Lep2]{Lep$2$} \textsc{Emmanuel Lepage}, \textit{Tempered fundamental group and graph of the stable reduction}, The Arithmetic of Fundamental Groups.\\

\bibitem[Lep3]{Lep$3$} \textsc{Emmanuel Lepage}, \textit{Resolution of non-singularities for Mumford curves}, Publications of the Research Institute for Mathematical Sciences \textbf{49} ($2013$), pp. $861-891$.\\

\bibitem[Lep4]{Lep$4$} \textsc{Emmanuel Lepage}, \textit{Tempered fundamental group}, in \textit{Geometric and differential Galois theories}, Séminaires et Congrès \textbf{27}, SMF, ($2013$), pp. $93-113$.\\


\bibitem[Mzk$2$]{M2} \textsc{Shinichi Mochizuki}, \textit{The Geometry of Anabelioids}, Publications of the Research Institute for Mathematical Sciences \textbf{40} ($2004$), pp. $819-881$.\\

\bibitem[Mzk$3$]{M3} \textsc{Shinichi Mochizuki}, \textit{Semi-graphs of Anabelioids}, Publications of the Research Institute for Mathematical Sciences \textbf{42} ($2006$).\\


\bibitem[PoSt]{PoSt} \textsc{Florian Pop} and \textsc{Jakob Stix} , \textit{Arithmetic in the fundamental group of a $p$-adic curve. On the $p$-adic section conjecture for curves.}, Journal für die Reine und Angewandte Mathematik \textbf{725} ($2017$), pp. $1-40$.\\


\end{thebibliography}
\end{document}